\documentclass[11pt]{article} 
\usepackage[utf8]{inputenc} 

\usepackage[margin=1in]{geometry} 
\usepackage{graphicx} 

\usepackage{dsfont} 

\usepackage{booktabs} 
\usepackage{array} 
\usepackage{paralist} 
\usepackage{verbatim} 
\usepackage{mathrsfs}
\usepackage{amssymb}
\usepackage{amsthm}
\usepackage{amsmath,amsfonts,amssymb}
\usepackage{esint}
\usepackage{graphics}
\usepackage{enumerate}
\usepackage{mathtools}
\usepackage{xfrac}
\usepackage{subcaption}
\usepackage{stmaryrd}
 \usepackage{mathabx}

\usepackage[dvipsnames]{xcolor}
\usepackage[colorlinks=true, pdfstartview=FitV, linkcolor=blue, citecolor=blue, urlcolor=blue]{hyperref}
\usepackage[normalem]{ulem}

\numberwithin{equation}{section}
\numberwithin{figure}{section}

\newtheorem{theorem}{Theorem}[section]

\newtheorem{corollary}[theorem]{Corollary}
\newtheorem{proposition}[theorem]{Proposition}
\newtheorem{lemma}[theorem]{Lemma}
\theoremstyle{definition}
\newtheorem{definition}[theorem]{Definition}

\newtheorem{remark}[theorem]{Remark}

\newcommand{\W}{\mathbb{W}}

\newcommand*{\N}{\ensuremath{\mathbb{N}}}

\newcommand*{\Z}{\ensuremath{\mathbb{Z}}}

\newcommand*{\R}{\ensuremath{\mathbb{R}}}

\newcommand{\eps}{\varepsilon}

\renewcommand*{\tilde}{\widetilde}

\renewcommand{\P}{\ensuremath{\mathbb{P}}}

\newcommand{\ep}{\eps}

\DeclareSymbolFont{boldoperators}{OT1}{cmr}{bx}{n}
\SetSymbolFont{boldoperators}{bold}{OT1}{cmr}{bx}{n}
\usepackage{accents}
\newcommand\thickbar[1]{\accentset{\rule{.45em}{.6pt}}{#1}}
\renewcommand{\bar}{\thickbar}

\def\XXint#1#2#3{{\setbox0=\hbox{$#1{#2#3}{\int}$}
\vcenter{\hbox{$#2#3$}}\kern-.5\wd0}}

\let\originalleft\left
\let\originalright\right
\renewcommand{\left}{\mathopen{}\mathclose\bgroup\originalleft}
\renewcommand{\right}{\aftergroup\egroup\originalright}

\newcommand{\indc}{\mathds{1}}

\newcommand{\cond}{\,|\,}

\newcommand{\E}{\mathbb{E}}

\renewcommand{\phi}{\varphi}

\newcommand{\solm}{\mathbb{X}}
\newcommand{\coupm}{\mathbb{Y}}
\newcommand{\volm}{\mathbb{U}}

\usepackage{titlesec}

\newcommand{\addperiod}[1]{#1.}
\titleformat{\section}
   {\centering\normalfont\Large}{\thesection.}{0.5em}{}
\titleformat*{\subsection}{\bfseries}
\titleformat{\subsubsection}[runin]
  {\normalfont\bfseries}
  {\thesubsubsection.}
  {0.5em}
  {\addperiod}
\titleformat*{\subsubsection}{\normalfont\itshape}
\titleformat*{\paragraph}{\bfseries}
\titleformat*{\subparagraph}{\large\bfseries}

\title{Sharp pathwise nonuniqueness for additive SDEs}

\author{Elias Hess-Childs\thanks{Department of Mathematical Sciences, Carnegie Mellon University.
{\footnotesize \href{mailto:ehesschi@andrew.cmu.edu}{ehesschi@andrew.cmu.edu}.}
}
\and 
Keefer Rowan\thanks{Institute of Mathematics, \'Ecole Polytechnique F\'ed\'erale de Lausanne.
{\footnotesize \href{mailto:keefer.rowan@epfl.ch}{keefer.rowan@epfl.ch}.}
}
}
\date{\today}

\usepackage[nottoc,notlot,notlof]{tocbibind}

\begin{document}

\maketitle

\begin{abstract}
    We construct a family of velocity fields demonstrating the sharpness of the classical Zvonkin--Veretennikov--Davie strong well-posedness by noise regime. We consider stochastic differential equations driven by Brownian noise with drift $u$ and show that for any $\alpha<0$, there exists a velocity field $u \in L^\infty_t C^\alpha_x$ that admits a unique weak solution but does not satisfy pathwise uniqueness (and hence has no strong solutions). This contrasts with the case $\alpha \geq 0$, for which the existence of a unique strong solution is guaranteed. The velocity field construction is random, and the proof essentially uses central limit theorem scaling through the Berry--Esseen theorem. We also give natural extensions to non-Brownian driving noises, including nonuniqueness for arbitrary driving noises with certain H\"older regularities and an analogous sharpness of the strong well-posedness by noise regime for fractional Brownian motions.
\end{abstract}


\section{Main results and background}

We are interested in uniqueness and nonuniqueness of solutions to the (generalized) stochastic differential equation
\begin{equation}
    \label{eq:main}
    \begin{cases}
    dX_t = u(t,X_t)\,dt + dW_t,
    \\X_0=y,
    \end{cases}
\end{equation}
where $y\in\R^d$. We suppose throughout that the drift $u \in L^1([0,T], C^0(\R^d))$ and that the \textit{driving noise} $W \in C^0([0,T],\R^d).$ The hypothesis $u \in L^1_t C^0_x$ allows us to work in a regime where solutions can be defined without regularization and without stochastic analysis. In particular, we use the following basic notion of solution. 

\begin{definition}
    We say that $X \in C^0([0,T],\R^d)$ is a \textit{Carath\'eodory solution} to~\eqref{eq:main} for drift $u\in L^1([0,T],C^0(\R^d))$ and (deterministic) $W\in C^0([0,T],\R^d)$ if for all $t\in[0,T]$,
    \[ X_t = y + \int_0^t u(s,X_s)\,ds + W_t -W_0.\]
\end{definition}

We call such solutions Carath\'eodory since by (a standard generalization of) the Carath\'eodory existence theorem for ODEs, there always exist Carath\'eodory solutions to~\eqref{eq:main} under our hypotheses on $u$ and $W$.

The groundbreaking work of~\cite{davie_uniqueness_2007} proved that for fixed $y \in \R^d$, if $u \in L^\infty_{t,x}$ and $W$ is a standard Brownian motion, then a.s.\ in $W$,~\eqref{eq:main} has a unique Carath\'eodory solution. The property of having a unique Carath\'eodory solution a.s.\ in the driving noise is known as \textit{path-by-path uniqueness}. The path-by-path result of \cite{davie_uniqueness_2007} is a strengthening of the classical works~\cite{zvonkin_transformation_1974,veretennikov_strong_1981}, which prove the existence and uniqueness of probabilistically strong solutions under the same hypotheses. The form of uniqueness proved by~\cite{zvonkin_transformation_1974,veretennikov_strong_1981} is known as \textit{pathwise uniqueness}. These different uniqueness notions (together with \textit{weak uniqueness}) will be more precisely defined and discussed in Subsection~\ref{subsec:uniqueness_def}. The phenomenon of~\eqref{eq:main} enjoying well-posedness for suitable driving noises despite well-posedness failing without a driving noise is known as \textit{regularization by noise.}

These foundational results have catalyzed a substantial literature on regularization by noise, which will be further discussed in Section~\ref{s:previous results}. However, the sharpness of~\cite{zvonkin_transformation_1974,veretennikov_strong_1981,davie_uniqueness_2007} in regularity spaces has long remained open, which we resolve with our main result. In particular, we show that for every $\rho>0$, there is a velocity field $u \in L^\infty_t C^{-\rho}_x$ such that pathwise (and thus also path-by-path) uniqueness fails for~\eqref{eq:main} with Brownian driving noise. We use the following definition for $C^{-\rho}_x$.

\begin{definition}
\label{def:generalized Holder spaces}
    For $\alpha\in [0,1)$, we take $\|f\|_{C^\alpha(\R^d)}$ as the usual H\"older norm. For $\rho \in (0,1]$, we take $\|f\|_{C^{-\rho}(\R^d)} := \inf_{\substack{\nabla \cdot g = f}} \|g\|_{C^{1-\rho}(\R^d)}.$ For general $s < -1$, we then recursively define $\|f\|_{C^{s}(\R^d)} := \inf_{\substack{\nabla \cdot g = f}} \|g\|_{C^{1+s}(\R^d)}$.\footnote{We note that for $s \not \in \Z$, these definitions of $C^s$ coincide with the homogeneous Besov space $\dot B^s_{\infty,\infty}$~\cite[Proposition 2.30 and Theorem 2.36]{bahouri_fourier_2011}.}
\end{definition}

By~\cite{flandoli_multidimensional_2017}, weak uniqueness holds in the regime $\rho \in (0,1/2)$. As such, this result gives the first (to our knowledge) example of pathwise nonuniqueness in the presence of weak uniqueness for additive noise equations; see Section~\ref{sec:contributions} for further discussion. We emphasize that, here and throughout, the velocity field $u$ is independent of the driving noise $W$.

\begin{theorem}
    \label{thm:main-brownian} For all $\rho\in (0,1/2)$, there exists a random velocity field $u: [0,1]  \times \R^2 \to \R^2$ such that we have the sure bound that $\|u(t,\cdot)\|_{C^{-\rho}(\R^2)} \leq 2$ for all $t\in [0,1]$ and $u \in L^1([0,1], C^0(\R^2)).$ Then, for $W$ an independent standard Brownian motion on $\R^2$, the velocity field $u$ has the following properties.
    \begin{enumerate}
        \item For any deterministic initial condition $y \in \R^2$, the ODE~\eqref{eq:main}---almost surely in $(u,W)$---has nonunique Carath\'eodory solutions on $[0,\ep]$ for every $\ep>0$,
        \item For all $y \in \R^2$, almost surely in $u$, the SDE~\eqref{eq:main} has a unique weak solution but exhibits pathwise nonuniqueness on $[0,\ep]$ for every $\ep>0$.
        \item The pathwise nonuniqueness is ``maximal'', in the sense that for all $y \in \R^2$, almost surely in $u$, for every $\ep>0,$ there exists a tuple $(X^1,X^2, W)$ such that $(X^1,W)$ and $(X^2,W)$ are both weak solutions to~\eqref{eq:main} and almost surely $X^1|_{[0,\ep]} \ne X^2|_{[0,\ep]}$.
    \end{enumerate}
\end{theorem}

\begin{remark}
The randomness of the velocity field is \textit{not essential to the statement of the theorem}: since all statements are almost sure in $u$, we in particular could fix a single deterministic velocity field $U \in L^\infty_t C^{-\rho}_x \cap L^1_t C^0_x$ such that the desired nonuniqueness holds. We include the randomness in the theorem statement to emphasize that the randomness of the velocity field is \textit{essential to the construction and proof}; see Section~\ref{sec:heuristics} for further discussion.
\end{remark}

\begin{remark}
    In the above result, we sometimes refer to~\eqref{eq:main} as an ODE and sometimes as an SDE. Throughout, when we call~\eqref{eq:main} an SDE we intend to emphasize that we are referring to stochastic solution theory of the equation---defined in Section~\ref{subsec:uniqueness_def}---which has certain additional informational/measurable aspects in addition to being a pointwise Carath\'eodory solution. We note however in these statements that we are taking $u$ to be fixed in a full measure set and are not further using the measurability structure in $u$. When we mean to disregard the additional informational structure in the driving noise---and in particular when $W$ is being taken to be deterministic---we refer to~\eqref{eq:main} as an ODE.
\end{remark}

\begin{remark}
    All velocity fields considered in this work will be \textit{divergence-free}: $\nabla \cdot u =0$. Thus in particular, the fields $u$ used in Theorem~\ref{thm:main-brownian} are divergence-free. In~\cite{grafner_weak_2024,hao_sdes_2025,kinzebulatov_sdes_2025}, taking $u$ divergence-free has allowed weak well-posedness results to be proved in a broader class of spaces. However, Theorem~\ref{thm:main-brownian} shows that the pathwise/path-by-path well-posedness theory in H\"older spaces is unaffected by whether $u$ is divergence-free. Some further discussion is given in Section~\ref{sec:contributions}.
\end{remark}

\begin{remark}
    It is essential that we are working in $d \geq 2$ for this result to hold.~\cite[Theorem 2.12]{butkovsky_weak_2025} proves pathwise uniqueness in $d=1$ for certain negative regularity velocity fields, in particular covering the regularities of the velocity fields in Theorem~\ref{thm:main-brownian}. 
\end{remark}

Beyond Brownian motion, the now substantial regularization by noise literature has proved uniqueness for~\eqref{eq:main} under a large variety of driving noises and regularity hypotheses. We will focus on fractional Brownian motion (fBm) driving noises in this work, but we note there is parallel literature which considers $\alpha$-stable processes~\cite{priola_pathwise_2012,priola_davies_2018,athreya_strong_2020,kremp_rough_2025}. Fractional Brownian motions are a family of Gaussian processes generalizing Brownian motion, indexed by the \textit{Hurst parameter} $H \in (0,1)$ (see Definition~\ref{def:fbm}). For our purposes, the primary facts of interest are that fBm with $H=1/2$ is just usual Brownian motion and that if $W$ is an fBm with Hurst parameter $H \in (0,1)$, then $W \in C^{H-}_t$ almost surely. That is, $W$ is in every $C^\beta$ space with $\beta < H$, analogously to Brownian motion with $H=1/2$. 

The path-by-path regularization by fBm is proved in~\cite{catellier_averaging_2016,galeati_noiseless_2021,galeati_solution_2025}. In particular, if $u \in L^\infty_t C^\alpha_x$ and $W$ is an fBm with Hurst parameter $H \in (0,1),$ then~\eqref{eq:main} has path-by-path (and hence pathwise) uniqueness whenever
\begin{equation}
\label{eq:hurst parameter condition}
\alpha > 1 - \frac{1}{2H}.
\end{equation}
Note that this includes $\alpha <0$; however in that case the more classical Carath\'eodory solution theory isn't available, and one has to be more careful in defining and proving existence of solutions. For $\alpha \geq 0$, path-by-path uniqueness is equivalent to a.s.\ uniqueness of Carath\'eodory solutions. 

Our next result goes in the direction of showing the sharpness of this threshold for all $H \in (0,1).$ For $H < 1/2$, we are not able to build velocity fields that also live in $L^1_tC^0_x$, thus precluding the Carath\'eodory solution theory. As such, we will need to modify the statement somewhat to avoid working with nonclassical solutions. We will first need the following definition.

\begin{definition}\label{def:sol}
    Suppose that $u \in C^\infty_{\mathrm{loc}}((0,T]\times\R^d)$ and $W \in C^0((0,T],\R^d)$ are fixed. Then for any $s>0, y \in \R^d,$ we denote the unique classical solution to the ODE
    \begin{equation*}
    \begin{cases}
    dX_t = u(t,X_t)\,dt + dW_t,
    \\X_s=y,
    \end{cases}
    \end{equation*}
    by $ t \mapsto X^{s,y, u, W}_t$, $t\ge s$.
\end{definition}

\begin{theorem}
    \label{thm:main-general}
    We have the following explosive separation estimates.
\begin{enumerate}
    \item    
    \label{item:general explosion}
    For all $\alpha \in \R$ with $\alpha <1/2$, there exists a random velocity field $v^\alpha: [0,1] \times \R^2 \to \R^2$ with law $\P$ such that we have the sure bound that $\|v^\alpha(t,\cdot)\|_{C^{\alpha}(\R^2)} \leq 2$ for all $t \in [0,1]$ and $v^\alpha \in C^\infty_{\mathrm{loc}}((0,1] \times \R^2)$. The velocity field $v^\alpha$ has the property that there exists a sequence of times $T^n$ with $\lim_{n \to \infty} T^n =0$ such that for any $\frac{1}{2(1-\alpha)} < \beta \leq 1$ and deterministic $W\in C^\beta([0,1],\R^2)$,
    \begin{equation*}
    \lim_{n \to \infty} \limsup_{\delta \to 0} \limsup_{|x-y| \to 0} \limsup_{m \to \infty} \P(|X^{T^m,x,v^\alpha, W}_{T^n} - X^{T^m,y,v^\alpha, W}_{T^n}| < \delta) =0. 
    \end{equation*}
    \item
    \label{item:brownian explosion}
    For all $\rho\in (0,1/2)$, there exists a random velocity field $u^\rho: [0,1]  \times \R^2 \to \R^2$ with law $\P$ such that we have the sure bound that $\|u^\rho(t,\cdot)\|_{C^{-\rho}(\R^2)} \leq 2$ for all $t\in [0,1]$ and $u^\rho \in L^1([0,1], C^0(\R^2)) \cap C^\infty_{\mathrm{loc}}((0,1] \times \R^2).$ The velocity field $u^\rho$ has the property that there exists a sequence of times $T^n$ with $\lim_{n \to \infty} T^n =0$ such that for any deterministic $W\in C^{\frac{1+\rho/8}{2+\rho/2}}([0,1],\R^2)$,
       \begin{equation*}
    \lim_{n \to \infty} \limsup_{\delta \to 0} \limsup_{|x-y| \to 0} \limsup_{m \to \infty} \P(|X^{T^m,x,u^\rho, W}_{T^n} - X^{T^m,y,u^\rho, W}_{T^n}| < \delta) =0. 
    \end{equation*}
\end{enumerate}
\end{theorem}

The vector fields in the above result are defined explicitly in Section~\ref{sec:definition of fields}. The condition on $\beta$ is equivalent to $\beta \in (0,1]$ and $\alpha < 1 - \frac{1}{2\beta},$ which is complementary to~\eqref{eq:hurst parameter condition}. The above result needs some interpretation. Since the velocity field is smooth away from the singular initial time, we have unique solutions to~\eqref{eq:main} when started at a positive time. Instead of proving a nonuniqueness result for solutions started at the initial time---which would require developing a somewhat delicate solution theory for $\alpha <0$---we prove an \textit{explosive separation estimate} for particles started arbitrarily close together at times very close to the initial time. Theorem~\ref{thm:main-general} can be read as saying no matter how close you start two particles together, if you flow them under~\eqref{eq:main} \textit{with the same driving noise} on the time interval $[\ep,1]$, then as you send $\ep\to 0$ they will a.s.\ instantaneously macroscopically separate.

It is exactly Theorem~\ref{thm:main-general}, Item~\ref{item:brownian explosion} that is the instability estimate used to prove the nonuniqueness in Theorem~\ref{thm:main-brownian}. One can view the estimates of Theorem~\ref{thm:main-general} as precluding \textit{any} kind of pathwise/path-by-path well-posedness theory, as it says that the flow is infinitely unstable to arbitrarily small perturbations at arbitrarily small times. For $\alpha <0$, we will refrain from further elaborating the result into a specific form of nonuniqueness at the initial time, due to the difficulties with defining solutions. For $\alpha \in [0,1/2)$, we will get the further nonuniqueness results of Corollary~\ref{cor:caratheodory general} and Corollary~\ref{cor:pathwise failure fBm}---analogous to those of Theorem~\ref{thm:main-brownian}---by combining the explosive separation estimates of Theorem~\ref{thm:main-general}, Item~\ref{item:general explosion} with the qualitative nonuniqueness theory given by Theorem~\ref{thm:general_non_path}.

The reason we need Item~\ref{item:general explosion} and Item~\ref{item:brownian explosion} in Theorem~\ref{thm:main-general} is that for Theorem~\ref{thm:main-brownian}, Corollary~\ref{cor:caratheodory general}, and Corollary~\ref{cor:pathwise failure fBm}, we will want to work with velocity fields that live in $L^1_t C^0_x$. For Corollary~\ref{cor:caratheodory general} and Corollary~\ref{cor:pathwise failure fBm}, this comes for free from $\alpha\geq 0$ and $v^\alpha \in L^\infty_t C^\alpha_x$. For Theorem~\ref{thm:main-brownian}, we will have to work with velocity fields $u^\rho \in L^\infty_t C^{-\rho}_x$, so the $L^1_t C^0_x$ bound is not ``free''. As such, we will need to modify the fields to additionally satisfy this bound---as is further discussed in Remark~\ref{rem:regularity tech}.

The other important aspect of Theorem~\ref{thm:main-general} is that we make no assumptions on $W$ other than the H\"older regularity. We are not using any special structure of the driving noise to deduce this separation, we only use that it does not fluctuate \textit{too much} on small time scales. This aspect demonstrates that the proof does not rely at all on the specific stochastic or geometric structure of the driving noise; for each $\alpha,$ we construct a single (random) velocity field in $C^\alpha_x$ that causes explosive separation for any driving noise in $C^\beta$ with $\beta > \frac{1}{2(1-\alpha)}.$  

We now note the nonuniqueness corollaries of Theorem~\ref{thm:main-general} and Theorem~\ref{thm:general_non_path}, which are available for $\alpha \geq 0$. First we give a statement about Carath\'eodory nonuniqueness for arbitrary deterministic $W$, only assuming they live in certain H\"older spaces.

\begin{corollary}
\label{cor:caratheodory general}
    For all $\alpha \in [0,1/2)$, there exists a random velocity field $u: [0,1]  \times \R^2 \to \R^2$ such that we have the sure bound that $\|u(t,\cdot)\|_{C^{\alpha}(\R^2)} \leq 2$ for all $t\in[0,1]$. The velocity field $u$ has the property that for any deterministic $W \in C^\beta([0,1],\R^2)$ with $\beta > \frac{1}{2(1-\alpha)}$ and $y\in\R^2$, the ODE~\eqref{eq:main}---almost surely in $u$---has nonunique Carath\'eodory solutions on $[0,\ep]$ for every $\ep>0$. 
\end{corollary}

Corollary~\ref{cor:caratheodory general} immediately shows that path-by-path uniqueness does not hold when the driving noise is fBm with Hurst parameter $H \in (0,1)$ and $\alpha < 1 - \frac{1}{2H}$, thus proving sharpness in H\"older spaces (modulo the critical endpoint) of the path-by-path uniqueness theory. However, the failure of pathwise uniqueness is somewhat stronger. This failure, which requires using (in a fairly mild way) the informational structure of fBm, is given by the next result.

\begin{corollary}
    \label{cor:pathwise failure fBm}
    For all $\alpha \in [0,1/2)$, there exists a random velocity field $u: [0,1]  \times \R^2 \to \R^2$ such that we have the sure bound that $\|u(t,\cdot)\|_{C^{\alpha}(\R^2)} \leq 2$ for all $t\in[0,1]$.  Then, for $W$ an independent fractional Brownian motion on $\R^2$ with Hurst parameter $H\in(1/2,1)$ such that  $\alpha < 1 - \frac{1}{2H},$ the velocity field $u$ has the following properties.
    \begin{enumerate}
        \item  For any deterministic initial data $y \in \R^2$, for almost every $u$, the SDE~\eqref{eq:main} exhibits pathwise nonuniqueness on $[0,\ep]$ for every $\ep>0$. 
        \item This pathwise nonuniqueness is ``almost sure in $W$''. There exists a weak solution $(X,W)$ to the SDE~\eqref{eq:main} such that for every $\ep>0$, the conditional law of $X|_{[0,\ep]}$ given $W$ is, almost surely in $W$, supported on more than one $X|_{[0,\ep]}$ path. In particular, for almost every $(u,W)$, the ODE~\eqref{eq:main} has more than one Carath\'eodory solution on every time interval $[0,\ep]$ with $\ep>0$.
    \end{enumerate}
\end{corollary}

\subsection{Weak, pathwise, and path-by-path}\label{subsec:uniqueness_def}

Before continuing our discussion, let us give clear definitions of the various solution and uniqueness notions present for stochastic differential equations. We restrict our attention to the case of additive noise as it allows us to avoid stochastic integration theory, which adds an additional layer of complication. For our purposes, we will always suppose $u \in L^1_t C^0_x, W \in C^0_t$, allowing for the simple Carath\'eodory/integral equation interpretation of~\eqref{eq:main}. Throughout this section, we thus take $u$ fixed. 

In our definition we also allow for general driving noises, though a particularly relevant class will be (f)Bm. Let us define the law of these processes now.

\begin{definition}
\label{def:fbm}
    We say that $\W(dW) \in \mathcal{P}(C^0([0,T], \R^d))$ is the path measure for $d$-dimensional fractional Brownian motion (fBm) with Hurst parameter $H \in (0,1)$, if $W$ is a centered Gaussian process under $\W$ and for all $s, t\geq 0$,
    \[\int_{C^0([0,T], \R^d)} W_{t,i} W_{s,j} \W(dW) = \frac{1}{2} \delta_{ij} \big(t^{2H} + s^{2H} - |t-s|^{2H}).\]
    In particular $H=1/2$ corresponds to Brownian motion.
\end{definition}

\begin{remark}
    Although our discussion only covers the case of deterministic initial data and additive noise, it can easily be adapted, \textit{mutatis mutandis}, to independent random initial data $y=Y$ and, at least in the case of fBm driving noise, multiplicative noise $\sigma(t,X_t)\,dW_t$. The presence of random initial data that is independent of the driving noise is easily accommodated via conditioning. Multiplicative noise only adds the additional complication of defining what we mean for the pair $(X,W)$ to almost surely solve~\eqref{eq:main}. This requires a stochastic or rough solution theory. This solution theory however is well understood for (f)Bm driving noise---through It\^o integration, rough differential equations, stochastic sewing, etc.---and the complexity is orthogonal to the notion of weak solutions, strong solutions, and associated ideas of uniqueness. 
\end{remark}

We now define a weak solution to~\eqref{eq:main}.
\begin{definition}
\label{def:weak}
    Let $u \in L^1([0,T],C^0(\R^d))$, $y \in \R^d$, and $\W\in \mathcal{P}(C^0([0,T],\R^d))$. A \textit{weak solution} to the SDE~\eqref{eq:main} with driving noise $\W$ is a probability measure $\solm(dX,dW) \in \mathcal{P}(C^0([0,T],\R^d) \times C^0([0,T],\R^d))$ such that
    \begin{enumerate}
        \item \label{item:marginalizes correct} $\W(dW) = \solm(C^0([0,T],\R^d), dW)$.
        \item \label{item:non-anticipatory} With respect to the probability measure $\solm$, for every $t \in[0,T]$, $X|_{[0,t]}$ and $W$ are conditionally independent given $W|_{[0,t]}.$
        \item \label{item:solves equation} We have $\solm$-almost surely that for all $t \in [0,T]$, 
        \[X_t = y + \int_0^t u(s,X_s)\,ds + W_t -W_0.\]
    \end{enumerate}
\end{definition}

Items~\ref{item:marginalizes correct} and~\ref{item:solves equation} are rather transparent: the first asks that the driving noise $W$ has the distribution we want while the second asks that a.s.\ the SDE is being solved pathwise. These are clearly necessary conditions for any reasonable solution theory. Item~\ref{item:non-anticipatory} is the heart of the definition as it encodes a non-trivial informational condition: it demands that $X|_{[0,t]}$ does not know about the future of $W$ for $s \geq t$. In the Brownian case, since $(W_s - W_t)_{s \geq t}$ is simply independent of $W|_{[0,t]}$, this property can be written as an independence condition: $X|_{[0,t]}$ and $W|_{[t,T]} - W_t$ are independent. However, in the case of more general noise---such as for fBm---there isn't such a straightforward decomposition, making the conditional independence statement necessary.

We now give the definition of a strong solution, which obeys a stronger informational (or, equivalently, measurability) condition than a weak solution.
\begin{definition}
\label{def:strong}
    Let $u \in L^1([0,T],C^0(\R^d)),$ $y \in \R^d$, and $\W\in \mathcal{P}(C^0([0,T],\R^d))$. Then a \textit{strong solution} to the SDE~\eqref{eq:main} with driving noise $\W$ is a probability measure $\solm(dX,dW) \in \mathcal{P}(C^0([0,T],\R^d) \times C^0([0,T],\R^d))$ such that $\solm$ is a weak solution, and $X$ is measurable with respect to $W$. That is, $X$ is $\overline{\sigma(W)}^\solm$ measurable, where $\overline{\sigma(Y)}^\solm$ denotes the $\solm$-completed $\sigma$-algebra generated by $Y$. 
\end{definition}

Note that with the non-anticipation condition given by Item~\ref{item:non-anticipatory}, this is equivalent to asking that for all $t\in[0,T]$, there exists a Borel measurable function $F_t$ such that $\solm$-almost surely, $X|_{[0,t]} = F_t(W|_{[0,t]})$. Thus a strong solution is uniquely determined by the noise---in an adapted manner---in contrast to a weak solution which is influenced by the noise in an adapted manner but not necessarily uniquely determined by it. The above definitions clearly correspond to the standard definitions of weak and strong solutions for (f)Bm driving noises.

\begin{definition}
\label{def:sde defs}
    Let $u \in L^1([0,T],C^0(\R^d)),$ $y \in \R^d$, and $\W\in \mathcal{P}(C^0([0,T],\R^d))$. Then the SDE for~\eqref{eq:main} with driving noise $\W$ has:
    \begin{enumerate}
        \item \textit{weak existence} if there exists a weak solution,
        \item \textit{weak uniqueness} if for all weak solutions $\solm,\tilde \solm$, we have that $\solm = \tilde \solm$,
        \item \textit{strong existence} if there exists a strong solution,
        \item\label{item:pathwise uniqueness} and finally \textit{pathwise uniqueness} if for every measure $\coupm(dX^1,dX^2, dW)$ such that the marginals $\coupm(C^0([0,T],\R^d), dX^2, dW)$ and $\coupm(dX^1, C^0([0,T],\R^d), dW)$ are weak solutions, it holds that $X^1 = X^2$ $\coupm$-almost surely.
    \end{enumerate}
\end{definition}

The only part of this definition that requires explanation is the definition of pathwise uniqueness. Pathwise uniqueness says that if we have the triple $(X^1,X^2,W)$ such that (the laws of) $(X^1,W)$ and $(X^2,W)$ are both weak solutions, then we have that $X^1 = X^2$ almost surely. This should be morally thought of as saying that the SDE solution $X$ is uniquely determined (in a pathwise sense) by $W$, thus if the $W$ is the same for two weak solutions, they must be the same.

A particularly useful way to analyze weak and strong solutions is through disintegration of the joint measure $\solm(dX,dW)$ into conditional measures. We let $\pi^t : C^0([0,T],\R^d) \to C^0([0,t],\R^d)$ denote the restriction map, so that $\pi^t \gamma(s) = \gamma(s)$ for $s \in [0,t]$. The proof of the following facts can be found in Appendix~\ref{appen:weak_sol_theory}.

\begin{lemma}
\label{lem:disintegration}
    Let $u \in L^1([0,T],C^0(\R^d)),$ $y \in \R^d$, and $\W\in \mathcal{P}(C^0([0,T],\R^d))$. Let $\solm(dX, dW)$ satisfy Item~\ref{item:marginalizes correct} of Definition~\ref{def:weak}---so that its $W$ marginal is $\W$---and Item~\ref{item:solves equation} of Definition~\ref{def:weak}---so that $(X,W)$ solves~\eqref{eq:main} $\solm$-almost surely. Then let $\solm(dX\cond W) \W(dW) = \solm(dX,dW)$ be the disintegration of $\solm$ into conditional measures. Then
    \begin{enumerate}
        \item \label{item:weak disintegration} $\solm$ is a weak solution if and only if for all $t\in[0,T]$, $W\mapsto\pi^t_*\solm(d\tilde{X}\cond W)$ is $\overline{\sigma(W_{[0,t]})}^{\mathbb{W}}$ measurable,
        \item \label{item:strong disintegration} and $\solm$ is a strong solution if and only if $\solm$ is a weak solution and $\solm(dX\cond W)$ is $\mathbb{W}$-a.s. supported on a singleton (i.e. $\solm(dX\cond W)$ is a Dirac $\delta$). 
    \end{enumerate}
\end{lemma}

The characterization of a weak solution in Lemma~\ref{lem:disintegration} says that the relevant informational condition for a weak solution is that the conditional law of $X|_{[0,t]}$ given $W$ depends only on $W|_{[0,t]}$ and not on $W|_{(t,T]}$. The fact that for a weak solution $\solm$, $\solm(dX\cond W)$ may still be a non-trivial measure encodes the central property of weak solutions: they may have additional randomness even after conditioning on the driving noise. This is clearly exemplified by the classical Tanaka example of a weak but not strong solution~\cite{watanabe_stochastic_2000}.

In contrast, the characterization of a strong solution says that a weak solution is strong if and only if $W$ uniquely determines $X$, that is the conditional law of $X$ given $W$ is a.s.\ deterministic---a single $\delta$ mass. This encodes the central distinction between weak and strong solutions: strong solutions have no additional randomness beyond that which is given by the driving noise.

This characterization of strong solutions and our definition of weak solutions allows us to directly note the following characterization of pathwise uniqueness.

\begin{lemma}
\label{lem:Yamada-Watanabe}
       Let $u \in L^1([0,T],C^0(\R^d)),$ $y \in \R^d$, and $\W\in \mathcal{P}(C^0([0,T],\R^d))$. Then the SDE for~\eqref{eq:main} with driving noise $\W$ has pathwise uniqueness if and only if it has weak uniqueness and all weak solutions are strong solutions.
\end{lemma}

In fact, this lemma is equivalent to the Yamada--Watanabe theorem~\cite{yamada_uniqueness_1971} (see~\cite[Section 5.3]{karatzas_brownian_1998} for a textbook treatment). The ``dual Yamada--Watanabe theorem''~\cite{engelbert_theorem_1991}---that strong existence and weak uniqueness implies pathwise uniqueness---follows directly from our definitions. We provide the short and straightforward proof of Lemma~\ref{lem:Yamada-Watanabe} in Appendix~\ref{appen:weak_sol_theory}.

\begin{remark}
    Note that the definition of pathwise uniqueness doesn't require the coupling $\coupm$ to satisfy a similar non-anticipation/conditional independence property analogous to Item~\ref{item:non-anticipatory} of Definition~\ref{def:weak}. That is, we don't necessarily have that $(X^1|_{[0,t]}, X^2|_{[0,t]})$ and $W$ are conditionally independent given $W|_{[0,t]}$---even though $X^i|_{[0,t]}$ and $W$ are conditionally independent given $W|_{[0,t]}$ for $i =1,2$ (since $(X^i,W)$ is a weak solution). This definition is consistent with e.g.\ the one given in~\cite[Section 5.3]{karatzas_brownian_1998}. However, following similar logic to the proof of Lemma~\ref{lem:Yamada-Watanabe} one can readily verify that the notion given in Definition~\ref{def:sde defs}, Item~\ref{item:pathwise uniqueness} is equivalent to the (\textit{a priori} weaker) definition where one additionally requires that $(X^1|_{[0,t]}, X^2|_{[0,t]})$ and $W$ are conditionally independent under $\coupm$ given $W|_{[0,t]}$. 

    Once one adds this additional non-anticipation condition, the measures $\coupm$ being considered are precisely weak solutions to the $\R^{2d}$-valued SDE
    \[
        dY_t = U(t,Y_t)\,dt + dV_t,\]
    where $Y = (X^1,X^2), U(t,x^1,x^2):=(u(t,x^1),u(t,x^2)),$ and $V$ has law 
    \[\mathbb{V}(dV^1,dV^2)= \int_{C^0([0,1])} \delta_{W}(dV^1)\delta_{W}(dV^2)\mathbb{W}(dW).\]
    This is called the \textit{two-point equation} as it gives the evolution of two particles with the same drift and driving noise. Pathwise uniqueness for initial data $y \in \R^d$ is then exactly equivalent to weak uniqueness for the two-point equation for initial data $(y,y) \in \R^{2d}$. This is in fact precisely what we will use when proving Theorem~\ref{thm:general_non_path}. 

    Further, using similar logic to the proof of Lemma~\ref{lem:Yamada-Watanabe}, one can readily verify that for the two-point equation, weak uniqueness on the diagonal---i.e., weak uniqueness for all initial data of the form $(y,y)$ with $y \in \R^d$---is equivalent to weak uniqueness everywhere, namely for any data $(y^1,y^2) \in \R^{2d}$. Thus pathwise uniqueness (for the original SDE) is equivalent to weak uniqueness for the two-point SDE. This discussion can clearly be further generalized to the (analogously) defined $n$-point process for $n \in \N$ (or even the $\N$-point process for countably many initial particles). One can try to further extend to get ``stochastic flows'', which give coupled flows for all initial data. While this is possible in some weak form, these flows are not (\textit{a priori}) particularly nice due to the usual issues of combining uncountably many measure zero sets; stochastic flows only become properly well-behaved under some additional regularity. 
\end{remark}

\begin{remark}
    We note the work~\cite{jan_integration_2002} which also studies---using a fairly distinct theoretical framework---multiplicative noise SDEs for which there is weak uniqueness but not pathwise uniqueness (and hence no strong solutions). For those familiar with those works, we note that pathwise uniqueness is equivalent in their terminology to being a \textit{flow of maps}~\cite[Definition 6.1 and Lemma 6.5]{jan_integration_2002} and pathwise nonuniqueness to the solution being \textit{diffusive}~\cite[Definition 6.3 and Definition 6.4]{jan_integration_2002}. They consider more generally \textit{statistical solutions}---which correspond in our setting to weak solutions---given by \textit{flows of kernels}~\cite[Theorem 3.2]{jan_integration_2002}, which are precisely the conditional measures appearing in Lemma~\ref{lem:disintegration}.
\end{remark}

Finally, we give the last notion of uniqueness we will consider: path-by-path uniqueness.

\begin{definition}
Let $u \in L^1([0,T],C^0(\R^d))$, $y \in \R^d$, and $\W\in \mathcal{P}(C^0([0,T],\R^d))$. The ODE~\eqref{eq:main} with driving noise $\W$ has \textit{path-by-path uniqueness} if for $\W$-almost every driving noise $W$, there is a unique Carath\'eodory solution to~\eqref{eq:main}.
\end{definition}

Unlike weak and pathwise uniqueness, path-by-path uniqueness has no informational aspect to it. That is, there is no measurability hypothesis analogous to Item~\ref{item:non-anticipatory} in the definition of a weak solution, Definition~\ref{def:weak}. Since every weak solution is, by Item~\ref{item:solves equation} of Definition~\ref{def:weak}, also (almost surely supported on) a Carath\'eodory solution, it is clear that path-by-path uniqueness is stronger than pathwise uniqueness, which in turn, by Lemma~\ref{lem:Yamada-Watanabe}, is stronger than weak uniqueness. That is, we have the following chain of implications:
\[\text{path-by-path uniqueness}\Rightarrow \text{pathwise uniqueness}\Rightarrow \text{weak uniqueness}.\]

\subsection{Previous results in weak, pathwise, and path-by-path uniqueness}

\label{s:previous results}

\subsubsection{Brownian driving noise}

There is an extensive literature on proving various forms of uniqueness under (seemingly) optimal hypotheses when the driving noise is Brownian. As noted above, this subfield originates with~\cite{zvonkin_transformation_1974,veretennikov_strong_1981}, which prove pathwise uniqueness under the hypothesis that $u \in L^\infty_{t,x}$ and Brownian driving noise. This result was then strengthened to path-by-path uniqueness---a stronger uniqueness notion---by~\cite{davie_uniqueness_2007}. Additionally, there is the more recent thread in the literature proving weak uniqueness under relaxed regularity hypotheses. The works~\cite{flandoli_multidimensional_2017,zhang_heat_2018} prove weak existence and uniqueness in a variety of negative regularity settings, in particular when $u \in L^\infty_t C^{-s}_x$ for $s<1/2$, though of course in this setting some care needs to be taken in correctly defining weak solutions. In fact, under certain structural assumptions, the regularity can be relaxed even further:~\cite{delarue_rough_2016,cannizzaro_multidimensional_2018} allow for $C^{-2/3}_x$ drifts, under certain structural hypotheses allowing for the construction of certain paracontrolled objects and~\cite{grafner_weak_2024,hao_sdes_2025,kinzebulatov_sdes_2025} prove weak uniqueness (suitably defined) all the way down to $C^{-1}_x$ for divergence-free drifts. In summary, there is a robust positive theory showing weak uniqueness in negative regularity spaces and pathwise/path-by-path uniqueness in non-negative regularity spaces. Despite the thoroughness of the positive theory, in the regime of $L^\infty_tC^{-s}_x$ for $s \in (0,1/2)$, weak uniqueness is known but the problem of pathwise or path-by-path uniqueness remained open.

In addition to the theory developed in regularity spaces, there is a complementary theory in integrability spaces.\footnote{Many of the negative regularity weak uniqueness results described above actually develop weak well-posedness theory in \textit{both} negative regularity and $L^p$ (or Besov) integrability hypotheses. For simplicity, we refrain from further discussing the known results for $B^{-s}_{p,q}$-type regularity.} The foundational result in this direction is~\cite{krylov_strong_2005}, which relaxed the hypothesis of~\cite{zvonkin_transformation_1974,veretennikov_strong_1981} of $u \in L^\infty_{t,x}$ to prove pathwise uniqueness to $u \in L^q_t L^p_x$ with $(p,q)$ satisfying the Ladyzhenskaya--Prodi--Serrin \cite{ladyzhenskaya-regularity-1967,prodi_teorema_1959,serrin_interior_1962} condition (also known as the Krylov--R\"ockner condition) of $\frac{2}{q} + \frac{d}{p} < 1$. These results have been further improved on:~\cite{rockner_sdes_2023,rockner_sdes_2025} prove pathwise uniqueness in the critical case of equality in the Ladyzhenskaya--Prodi--Serrin condition and~\cite{anzeletti_path-by-path_2025} generalizes~\cite{davie_uniqueness_2007} to prove path-by-path uniqueness in the same setting as~\cite{krylov_strong_2005}. In contrast to the theory in regularity spaces, there is no parameter range in the integrability spaces for which weak uniqueness is known but pathwise uniqueness is not known. In fact, by~\cite[Remark 1.11]{galeati_solution_2025}, weak---hence also pathwise and path-by-path---uniqueness can fail for any $(p,q)$ with $\frac{2}{q} + \frac{d}{p} > 1$ and $p >d$.

\subsubsection{Fractional Brownian driving noise}

There is also a substantial literature on the case of fractional Brownian motion driving noise. For simplicity, we focus here only on the regularity space theory, ignoring the $L^q_t L^p_x$ developments.~\cite{nualart_regularization_2002} proved a version of~\cite{zvonkin_transformation_1974,veretennikov_strong_1981} for fractional Brownian driving noises, showing pathwise uniqueness for $C^\alpha_x$ drifts whenever $\alpha > (1 - \frac{1}{2H}) \lor 0$ for Hurst parameter $H \in (0,1).$ \cite{catellier_averaging_2016,galeati_noiseless_2021,galeati_solution_2025} improve the result to path-by-path uniqueness for $L^\infty_t C^\alpha_x$ whenever $\alpha > 1 - \frac{1}{2H}$, allowing $\alpha <0$. Note that for $H =1/2$, these results respectively coincide with~\cite{zvonkin_transformation_1974,veretennikov_strong_1981}
 and~\cite{davie_uniqueness_2007} (up to the critical case of equality). \cite{butkovsky_weak_2025} shows weak uniqueness for solutions to the equation with an autonomous drift in $C^\alpha_x$ and driven by fBm with Hurst parameter $H \in (0,1/2)$ with $\alpha > \frac{1}{2} - \frac{1}{2H}$. For $H =1/2$, this corresponds to the $-1/2$ regularity threshold proved in~\cite{flandoli_multidimensional_2017,zhang_heat_2018}.

\subsection{Comparison with spontaneous stochasticity}

The main results of this work give nonuniqueness to certain finite dimensional differential equations. Another area in which such a phenomenon is a central area of interest is the study of spontaneous stochasticity in the fluid and passive scalar turbulence literature. Fixing some divergence-free velocity field $u \in L^\infty_t C^\alpha_x$ for some $\alpha \in [0,1)$, for all $\kappa>0$, let $X_t^\kappa$ solve
\begin{equation}
\label{eq:spontaneous stochasticity}
\begin{cases}
    dX^\kappa_t = u(t,X^\kappa_t)dt + \kappa dW_t\\
    X^\kappa_0 =y,
\end{cases}\end{equation}
where $W$ is a standard Brownian motion. Roughly, spontaneous stochasticity~\cite{bernard_slow_1998,gawedzkiKrzysztofGawedzkiSoluble2008} refers to the persistence of noise in the vanishing noise limit, e.g.\ $\limsup_{\kappa \to 0^+} \mathrm{Var}_W(X^\kappa_1) >0$, where we write $\mathrm{Var}_W$ to emphasize that the probability is coming purely from the driving Brownian motion $W_t$. Since $X^\kappa_t$ concentrates on Carath\'eodory solutions to the ODE without driving noise as $\kappa \to 0^+$, spontaneous stochasticity in particular implies a form of ODE nonuniqueness. However, the nonuniqueness phenomenon of spontaneous stochasticity is of a rather different character than pathwise nonuniqueness for the SDE. Since $\alpha \in [0,1)$, for all $\kappa>0$ the SDE~\eqref{eq:spontaneous stochasticity} has path-by-path and hence pathwise uniqueness by~\cite{davie_uniqueness_2007}, the ODE nonuniqueness is only appearing in the $\kappa \to 0^+$ limit. 

In the setting considered in this work, we are interested in pathwise nonuniqueness to the SDE without sending the noise coefficient to zero. As such, pathwise nonuniqueness is both philosophically and technically rather different than spontaneous stochasticity. Nonetheless, the ideas developed in the spontaneous stochasticity literature certainly helped inspire the construction of this work. In particular, the works~\cite{drivas_anomalous_2022,colombo_anomalous_2023,elgindi_norm_2024} prove spontaneous stochasticity (or really the equivalent property of \textit{anomalous dissipation} by~\cite{drivas_lagrangian_2017}) using rescaled alternating shear flows, which were in turn inspired by the singular mixing flows of~\cite{aizenman_vector_1978,depauw_non_2003}. Our construction in some ways resembles a randomized version of this alternating shear flow construction---though the actual mechanism of nonuniqueness is rather different due to the CLT scaling as discussed in Section~\ref{sec:heuristics}. In fact, this CLT scaling relates our construction to another spontaneous stochasticity example, the Kraichnan model~\cite{kraichnanSmallScaleStructure1968,bernard_slow_1998,rowan_anomalous_2024}. 

\subsection{Open problems}

We note some open problems that would constitute nice additions to the literature and help complete the picture of regularization by noise in H\"older spaces. The first such result would be extending the SDE weak well-posedness theory of~\cite{butkovsky_weak_2025} to time dependent drifts and $H \in (1/2,1)$. This would, together with the current work, show that for $H \in (1/2,1)$ there is an (optimal) regime of weak but not pathwise uniqueness. It would additionally show that the solutions constructed in Corollary~\ref{cor:pathwise failure fBm} are the unique weak solutions for the associated SDE, but are not strong solutions, just as in the Brownian case.

There is also the problem of proving pathwise nonuniqueness for $H \in (0,1/2)$. We believe this should follow by an appropriate combination of the weak solution theory of~\cite{butkovsky_weak_2025} with the explosive separation estimate of Theorem~\ref{thm:main-general}, Item~\ref{item:general explosion}.

There is the question of extending the current work to allow for $\alpha$-stable L\'evy processes. We believe in particular that Theorem~\ref{thm:main-general} should be provable under the hypothesis that the driving noise is c\`adl\`ag with appropriate bounds on the jump size/density.  

Finally, there is the much more ambitious problem of understanding a similar phenomenon of pathwise nonuniqueness in the presence of weak uniqueness for the generic negative regularity fields considered in~\cite{armstrong_superdiffusion_2026}.

\subsection{Acknowledgments}
The authors would like to thank Oleg Butkovsky and Lucio Galeati for stimulating discussions. The first author was partially supported by NSF grant DMS-2342349.

\section{The heuristic argument and discussion}\label{sec:heuristics}

In this section, we first present the primary heuristic behind our argument. We then further discuss our results: first explaining the structure of the argument, then discussing the general weak SDE theory allowing us to pass from instability estimates like those of Theorem~\ref{thm:main-general} to nonuniqueness statements like those of Theorem~\ref{thm:main-brownian}. Finally, we provide some discussion and interpretation of pathwise nonuniqueness in the setting of weak uniqueness, giving a comparison to the DiPerna--Lions theory.

\subsection{A heuristic computation}

\label{sec:heuristic-comp}
The central idea is that (qualitative) nonuniqueness is downstream of (quantitative) separation estimates. In particular, what we want to show is that, if we take any two points $x,y$ with $x \ne y$, and flow them under the ODE with the same driving noise, then they will separate to a macroscopic amount---independent of their initial separation---at unit time. This will be shown by constructing, for all $n \in \N$, a velocity field that takes particles separated at scale $2^{-n}$ and separates them to scale $2^{-n+1}$ in time $\tau_n$. If $\sum_n \tau_n <\infty$, then we can combine these velocity fields so that particles that are infinitesimally initially separated at time $t=0$ are macroscopically separated at unit time. In order to get the desired regularity, we need that the velocity field constructed for $n \in \N$ obeys $n$-uniform $L^\infty_t C^\alpha_x$ bounds. We emphasize also that the two particles will share the same driving noise throughout, so we need to ensure the driving noise does not interfere with the separation estimate. We will assume that the driving noise $W$ is in $C^\beta([0,1])$ for some $\beta$ fixed; it is this assumption we will rely on to ensure that the driving noise does not damage the separation estimate.

So, let us fix $n \in \N, \alpha  \in \R, \beta \in (0,1)$, and let $W \in C^\beta([0,1])$ be frozen. We then want to build a velocity field $u(t,x)$ (independent of $W$) and times $\tau_n$ such that $\|u\|_{L^\infty_t C^\alpha_x}$ is bounded uniformly in $n$, $\sum_n \tau_n <\infty$, and if $x, y$ are such that $|x-y| \geq 2^{-n}$ then---letting $X_t,Y_t$ solve~\eqref{eq:main} with $X_0=x, Y_0 =y$---we have that $|X_{\tau_n} - Y_{\tau_n}| \geq 2^{-n+1}$.

Let us take $u$ to be a fixed shear flow on some time scale $\delta_n$ (it will turn out that $\delta_n \ll \tau_n$) and we will take it to be fluctuating in space on length scale (about) $2^{-n}$, in particular we will take it to be mean-zero periodic with some periodicity of about $2^{-n-2}$. In order to maintain the proper $C^\alpha_x$ regularity, we are then forced to take the magnitude of the velocity field to be $\approx 2^{-\alpha n}$. On the time scale $\delta_n$, $W_t$ will change on the length scale $\delta_n^\beta$, using that $W \in C^\beta$. We want that $\delta_n^\beta \ll 2^{-n}$, as otherwise, on the time scale $\delta_n$, the driving noise $W$ will drag the particles $X_t, Y_t$ over many periods of the velocity field $u$. As $W$ drags our particles over a period of the velocity field, the ``hit'' the velocity field gives the particles is (possibly) very small, since the field is zero-mean over the period, so potentially the effect of the velocity field would perfectly average. One could hope to use the probabilistic structure of $W$ to ensure that the averaging isn't perfect and get a nontrivial ``hit'' from the velocity field that way. That strategy---which is essentially that of homogenization e.g.\ as used in a somewhat similar fashion in~\cite{armstrong_anomalous_2025,burczak_anomalous_2023_fixed,chatzigeorgiou_gaussian_2025,armstrong_superdiffusive_2024,armstrong_superdiffusion_2026,burczak_scalar_2026}---departs meaningfully from our current strategy since we don't want to assume any specific structure on $W$: Theorem~\ref{thm:main-general} works for any $W \in C^\beta$. Instead, the restriction that $\delta_n^\beta \ll 2^{-n}$ ensures that such cancellation over a period cannot occur as the driving noise moves the particle only a trivial fraction of the periodicity length scale on the time scale $\delta_n$.

So, taking $\delta_n \approx 2^{-\beta^{-1}n }$ so $\delta_n^\beta \ll 2^{-n}$ (provided we take the prefactor on $2^{-\beta^{-1} n}$ small enough), we will now pretend the driving noise does not exist, since it essentially does not move the particle in the time scale under study. We want to compute how far apart $X_t, Y_t$ move on the time scale $\delta_n$. Since $|X_t - Y_t| \geq 2^{-n}$ and $u$ is fluctuating on length scale $2^{-n-2}$ and has magnitude $2^{-\alpha n}$, the reasonable guess is that $X_{\delta_n} - Y_{\delta_n} = X_0 - Y_0 + R$ where $R \approx 2^{-\alpha n} \delta_n \approx 2^{-(\alpha+ \beta^{-1}) n}$. We want to separate $|X_t - Y_t| \geq 2^{-n+1}$, so we'd be (more or less) done if $2^{-(\alpha+ \beta^{-1}) n} \gg 2^{-n}$, or $\alpha+ \beta^{-1} \leq 1$. However, in the Brownian case $\beta = 1/2-$, so this would become $\alpha < -1$.\footnote{The appearance of $\alpha <-1$ is strongly related to the $-1$ threshold of weak uniqueness for Brownian drivers in divergence-free fields. Currently in the argument we are trying to work on a time scale on which the driving noise does essentially nothing and get all the separation in that time scale. However, if you could really do that, it turns out you can construct distinct weak solutions rather than just having pathwise nonuniqueness, hence the restriction of $\alpha<-1$.} As our goal then is to only have $\alpha <0$, we cannot just stop after time $\delta_n$.

It is at this point where the random nature of the velocity field becomes essential. The idea is to choose a new velocity field on every time interval of size $\delta_n$ a total of $N_n$ many times, taking $\tau_n = N_n \delta_n$. Then 
\[X_{\tau_n} - Y_{\tau_n} = X_0-Y_0 + \sum_{j=1}^{N_n} R_j,\]
where the $R_j$ are the ``hits'' of the velocity field on each $\delta_n$ time interval. Once we have that $|\sum_{j=1}^{N_n} R_j| \gg 2^{-n}$ we can conclude. The issue is how to lower bound the sum of many small hits without having specific control of $W$. The answer is given by the central limit theorem and statistical independence. By taking $u$ to be random and independently selected on each time interval, we can apply a quantitative CLT to conclude that
\[\Big|\sum_{j=1}^{N_n} R_j\Big| \approx N_n^{1/2} |R| \approx N_n^{1/2} 2^{-(\alpha +\beta^{-1}) n},\]
where we plugged in the typical size of $R$ as computed above. We see that as we make $N_n$ large, the better the separation becomes. Our restriction on $N_n$ is that we want 
\[\sum_n \tau_n = \sum_n N_n \delta_n \approx \sum_n N_n 2^{-\beta^{-1} n} <\infty.\]
Thus we must take $N_n \ll 2^{\beta^{-1} n}$. Pretending we could take $N_n \approx 2^{\beta^{-1}n}$ (which we can do up to a very mild correction), we get that 
\[\Big|\sum_{j=1}^{N_n} R_j\Big| \approx 2^{-(\alpha + \beta^{-1}/2)n}.\]
In order to get the separation we want on the time scale $\tau_n$, we need then that $2^{-(\alpha + \beta^{-1}/2)n} \gg 2^{-n}$, or 
\begin{equation}
\label{eq:alpha beta relation heuristic}
\alpha  < 1 - \frac{1}{2\beta},
\end{equation}
which is equivalent to the condition of Theorem~\ref{thm:main-general} and for $\beta = 1/2-$, equivalently to $\alpha < 0$ as in Theorem~\ref{thm:main-brownian}.

In summary, we break up $[0,1]$ into time intervals (with size depending on $\beta$, the H\"older regularity of the driving noise) on which the driving noise does not move the ODE solutions an appreciable fraction of the fluctuation length scale $2^{-n}$. On each of these time intervals, the additional separation is much too small to immediately go from particles separated on scale $2^{-n}$ to scale $2^{-n+1}$, so we need to sum up across many such short time intervals. To control such a sum, we select the velocity independently across each short time interval and use a quantitative CLT to control the total effect on the separation of the particles. The more short time intervals, the better the separation, but we need to make sure that the total amount of time spent on length scale $2^{-n}$ is summable in $n$, which creates a restriction on how many short time intervals we can use. This restriction, together with the $L^\infty_t C^\alpha_x$ constraint together become~\eqref{eq:alpha beta relation heuristic}, which relates the regularity of the velocity field to the regularity of the driving noise.

\subsection{Discussion of results}
\label{sec:contributions}

\subsubsection{Explosive separation and pathwise nonuniqueness}

The above heuristic is made precise in Sections~\ref{sec:one scale}--\ref{sec:multiscale}, using the velocity fields constructed in Section~\ref{sec:definition of fields}. We will build our velocity fields out of a single shear flow at each time, which will slightly modify the heuristic argument; see the beginning of Section~\ref{sec:one scale}. Our argument vitally uses the Berry--Esseen theorem as a quantitative CLT to capture the scaling that we argued roughly above. The use of CLT scaling in order to get the desired cascade of separations is quite different than the other strategies (of which we are aware) for constructing nonunique ODE solutions. However, this is seemingly necessary in this setting, since we have absolutely no control of what the driving noise $W$ is doing, besides that it does not fluctuate too much. For a deterministic velocity field, it seems we could build a $W$ to maximally conspire to suppress the separation growth. For a random velocity field, the independence across time (and independence from $W$) ensures that $W$ almost surely cannot conspire to suppress the separation.

\begin{remark}
\label{rem:measure notation}
    As we begin to enter the technical core of the paper, let us note the following notational convention. We will be working with a variety of probability measures in this work, the most basic of which is the measure on velocity fields and the measure on driving noises, which we will always take to be independent from each other. In Sections~\ref{sec:definition of fields}--\ref{sec:multiscale}, we will be taking the driving noise to be fixed and control probabilities in the velocity field measure. In Section~\ref{sec:quantitative to qualitative} and Appendix~\ref{appen:weak_sol_theory}, we will (primarily) take $u$ fixed and control probabilities in the driving noise measure. In order to avoid confusion, we refer to the measure on velocity fields $u$ as $\volm$ and the measure on the driving noises $W$ as $\W$. We will also have the weak solution measures, which are naturally coupled to the driving noise (and implicitly depend on the velocity field, but when we are discussing weak solution measures, we will always have the velocity field fixed). We denote the measure on weak solutions $(X,W)$ by $\solm.$ Finally, when considering pathwise uniqueness statements, we will work with coupled solutions $(X^1,X^2,W)$; we refer to the measure on these coupled solutions as $\coupm$.
\end{remark}

The consequence of Sections~\ref{sec:definition of fields}--\ref{sec:multiscale} will be (a quantitative version of) Theorem~\ref{thm:main-general}, which gives an asymptotically almost sure macroscopic separation of particles started arbitrarily close together. We refer to this phenomenon as explosive separation, which we give the following precise definition.

\begin{definition} \label{def:explosive separation}
A random velocity field $u$ with measure $\volm \in \mathcal{P}(C^\infty_{\mathrm{loc}}((0,1] \times \R^d))$ is \textit{explosively separating} for a sequence of times $T^n\rightarrow 0^+$ and a fixed $W \in C^0([0,1],\R^d)$ if 
\[ \lim_{n \to \infty} \limsup_{\delta \to 0} \limsup_{|x-y| \to 0} \limsup_{m \to \infty}\volm(|X^{T^m,x,u, W}_{T^n} - X^{T^m,y,u, W}_{T^n}|< \delta) =0.\]
\end{definition}

In order to conclude Theorem~\ref{thm:main-brownian}, Corollary~\ref{cor:caratheodory general}, and Corollary~\ref{cor:pathwise failure fBm} from the explosive separation estimates of Theorem~\ref{thm:main-general}, we will need a qualitative theory that takes (qualitative) explosive separation and produces (qualitative) nonuniqueness. We want to prove pathwise nonuniqueness, which is a property of weak solutions and hence uses the specific informational structure of the driving noise path measure. We will need a very mild regularity condition on the driving noise path measures, which we define next, recalling the restriction notation $\pi^t$ used in Section~\ref{subsec:uniqueness_def}.

\begin{definition}\label{def:regular_driving_noise}
Let $\W\in \mathcal{P}(C^0([0,1],\R^d))$. Then $\W$ is a \textit{regular driving noise} if for all $t\in[0,1]$, there exists a representative of the conditional law of $W$ given $\pi^t W$, denoted by $\W(d\tilde{W}\cond \pi^tW)$, such that $\pi^tW\mapsto \mathbb{W}(d\tilde W\cond \pi^tW)$ is continuous in the weak topology on probability measures. 
\end{definition}

\begin{remark}\label{rem:regular_driving_noise}
(f)Bm is clearly a regular driving noise for any $H \in (0,1)$ (e.g.\ by~\cite[Theorem 3.1]{sottinen2017prediction}). We also have that a deterministic path $W\in C^0([0,1],\R^d)$ defines a regular driving noise by taking the path measure to be $\delta_{W}.$ It is also clear that solutions to (sufficiently regular) SDE (or RDE) driven by (f)Bm are regular driving noises. Finally, the sum of independent regular driving noises is a regular driving noise, for example sums of independent fBms with different Hurst parameters. 
\end{remark}

With Definition~\ref{def:explosive separation} and Definition~\ref{def:regular_driving_noise} in hand, we are ready to state our main result that relates explosive separation estimates to pathwise (and Carath\'eodory) nonuniqueness.

\begin{theorem}\label{thm:general_non_path}
Suppose that for a regular driving noise with measure $\mathbb{W}\in \mathcal{P}(C^0([0,1],\R^d))$, an independent random velocity field $u$ with measure $\volm \in \mathcal{P}\big(C^\infty_{\mathrm{loc}}((0,1] \times \R^d) \cap L^1([0,1], C^0(\R^d))\big)$, and a sequence of times $T^n\rightarrow 0^+$, $u$ is explosively separating for $T^n$ and $\mathbb{W}$-almost every $W$. Then, for every $y\in\R^d$, $\volm$-almost every $u$, the SDE~\eqref{eq:main}
with driving noise $\mathbb{W}$ exhibits pathwise nonuniqueness. More precisely---$\volm$-almost surely in $u$---we have the following statements.
\begin{enumerate}
    \item \label{item:high prob pathwise}
    The pathwise nonuniqueness happens instantaneously and with arbitrarily high probability: for all $\ep,\gamma>0,$ there exists $\coupm(dX^1,dX^2, dW)$ so that the marginals $\coupm(C^0([0,1],\R^d), dX^2, dW)$ and $\coupm(dX^1, C^0([0,1],\R^d), dW)$ are weak solutions to~\eqref{eq:main}, and $\coupm(X^1|_{[0,\ep]} = X^2|_{[0,\ep]}) \leq \gamma$.
    \item \label{item:as non determin} The nonuniqueness is ``almost sure in $W$'': there exists a weak solution $\solm(dX,dW)$ to~\eqref{eq:main} such that for all $\ep>0$, the conditional measure of $X|_{[0,\ep]}$ given $W$---$\pi^\ep_* \solm(d\tilde{X} \cond W)$---is $\W$ almost surely supported on more than one $X|_{[0,\ep]}$ path. In particular, for almost every $(u,W)$,~\eqref{eq:main} has more than one Carath\'eodory solution on every time interval $[0,\ep]$ with $\ep>0$.
    \item \label{item:weak unique as pathwise nonunique} Under weak uniqueness, the pathwise nonuniqueness is instantaneous and ``almost sure'': if~\eqref{eq:main} admits a unique weak solution, then for every $\ep>0,$ there exists a measure on the coupled process $(X^1,X^2,W)$, $\coupm(dX^1,dX^2, dW)$, such that the marginals $\coupm(C^0([0,1]), dX^2, dW)$ and $\coupm(dX^1, C^0([0,1]), dW)$ are weak solutions to~\eqref{eq:main} and $\coupm(X^1|_{[0,\ep]} = X^2|_{[0,\ep]}) =0$.
\end{enumerate}
\end{theorem}

Theorem~\ref{thm:general_non_path} will be proved in Section~\ref{sec:quantitative to qualitative}. We note that Item~\ref{item:weak unique as pathwise nonunique} of Theorem~\ref{thm:general_non_path}, Item~\ref{item:brownian explosion} of Theorem~\ref{thm:main-general}, and the weak uniqueness for SDE driven by Brownian motion with $u \in L^\infty_t C^s_x$ for $s \in (-1/2,0)$ given by~\cite{flandoli_multidimensional_2017} immediately imply Theorem~\ref{thm:main-brownian}. We also have that Items~\ref{item:high prob pathwise} and~\ref{item:as non determin} of Theorem~\ref{thm:general_non_path} and Item~\ref{item:general explosion} of Theorem~\ref{thm:main-general} imply Corollary~\ref{cor:caratheodory general} and Corollary~\ref{cor:pathwise failure fBm}.

We thus see that Theorem~\ref{thm:general_non_path} provides a robust tool for taking explosive separation estimates of the form given by Definition~\ref{def:explosive separation} and generating sharp and maximal statements of stochastic nonuniqueness. Thus Theorem~\ref{thm:general_non_path} justifies our general approach to proving pathwise and path-by-path nonuniqueness, which is to prove instability estimates that, in their qualitative form, become explosive separation estimates as in Definition~\ref{def:explosive separation}. We note though that---as will be clear in Sections~\ref{sec:one scale} and~\ref{sec:multiscale}---it is important to propagate quantitative separation estimates scale-by-scale in our multiscale iteration. Only at the end, once we have finished the iteration and concluded the quantitative separation estimate (e.g.\ Proposition~\ref{prop:brownian separation}) can we soften it to its qualitative form (e.g.\ Theorem~\ref{thm:main-general}, Item~\ref{item:brownian explosion}).

\begin{remark}
\label{rem:regularity tech}
    One technical point of difficulty---which is the origin of the velocity fields $u^\rho$ in addition to the velocity fields $v^\alpha$---is that for $\alpha<0$, the velocity field we construct in Section~\ref{sec:definition of fields} following the heuristic given in Section~\ref{sec:heuristic-comp} will be in $L^\infty_t C^\alpha_x$ but not $L^1_t C^0_x$. While $v^\alpha$ for $\alpha<0$---by Theorem~\ref{thm:main-general}, Item~\ref{item:general explosion}---will a.s.\ obey explosive separation estimates when given Brownian motion driving noise, we won't be able to directly conclude the desired pathwise and path-by-path uniqueness from Theorem~\ref{thm:general_non_path}. However, since we can take $\alpha<0$ arbitrarily close to $0$ and we have $v^\alpha \in L^\infty_t C^\alpha_x$, we will have that $v^\alpha$ is very close to being in $L^1_t C^0_x$. If we could gain a bit from the measure term in time---since the velocity field will only be large for a short amount of time---we could hope to get that actually $v^\alpha \in L^1_t C^0_x$. However, as constructed---which is seemingly necessary to get the optimal constraint on $\beta$ with respect to $\alpha$---the gain in time measure is essentially negligible. As such, we need to introduce the velocity fields $u^\rho$, for which there is a better gain in the time measure and so that $u^\rho \in L^1_t C^0_x \cap L^\infty_t C^{-\rho}_x.$ As is stated in Theorem~\ref{thm:main-general}, Item~\ref{item:brownian explosion}, $u^\rho$ will give explosive separation for any driving noise in $C^{1/2-}_t$, hence in particular for Brownian motion.
\end{remark}

\subsubsection{Weakly unique, pathwise nonunique, and DiPerna--Lions theory}

Weak uniqueness in the context of pathwise nonuniqueness is a subtle phenomenon. Weak uniqueness, as is made clear by Lemma~\ref{lem:disintegration}, tells us there is (almost surely) a unique way to choose the conditional law of $X$ given the driving noise $W$---$\solm(dX \cond W)$---in a way that respects the informational (or, equivalently, measurability) structure of the problem---that is such that $X|_{[0,t]}$ does not nontrivially depend on the future of the driving noise $W|_{[t,T]}$. 

However, when the unique weak solution isn't strong---or equivalently, by Lemma~\ref{lem:Yamada-Watanabe}, there is pathwise nonuniqueness---then $X$ is not uniquely determined by $W$: the conditional measure $\solm(dX \cond W)$ is not a Dirac mass and so $X$ isn't deterministic conditionally on $W$. Thus there is ``additional randomness'' beyond that supplied purely by the driving noise. The rather subtle aspect of this problem is how one is prevented from constructing nonunique weak solutions using the pathwise nonuniqueness. Since $\solm(dX \cond W)$ is supported on solutions to the integral equation---and, with a nontrivial $W$ probability, is supported on multiple such solutions---it seems one could simply take $\solm^1(dX \cond W)$ to be a probability measure on ``half'' of the support of $\solm(dX \cond W)$ and $\solm^2(dX \cond W)$ to be a probability measure on the ``other half''. In this way, we would get $\solm^1(dX, dW) = \solm^1(dX \cond W) \W(dW)$ and  $\solm^2(dX, dW) = \solm^2(dX \cond W) \W(dW)$ which are mutually singular measures which straightforwardly satisfy Item~\ref{item:marginalizes correct}---since the $W$ marginal law is $\W$---and Item~\ref{item:solves equation}---since the measures are built out of solutions to the integral equation---of Definition~\ref{def:weak}. All one has to do is make this selection of $\solm^1(dX \cond W)$ and $\solm^2(dX \cond W)$ in an appropriately adapted way and Item~\ref{item:non-anticipatory} would be satisfied as well, thus constructing distinct weak solutions and proving weak nonuniqueness.

The above argument therefore proves that under weak uniqueness and pathwise nonuniqueness, there is no appropriately adapted way of splitting the conditional measure $\solm(dX \cond W)$ into mutually singular components: the informational structure of the problem prevents any such decomposition. We believe that the additive driving noise setting considered here dramatically simplifies this rather subtle phenomenon of weak uniqueness but pathwise nonuniqueness, in contrast to examples where the pathwise nonuniqueness is induced by a non-smooth multiplicative noise coefficient as in the Tanaka example~\cite{watanabe_stochastic_2000}.\footnote{We note that there is also the additive noise example of Tsirelson~\cite{Tsirelson_example-1975}, however this uses a path-dependent drift. Since a path dependent drift can detect e.g.\ the entire past path of the Brownian motion, we believe that this phenomenon is of a rather different character.} In the setting of multiplicative noise, it seems on first reading that some of this strange behavior---in particular the nonexistence of an appropriately adapted decomposition of the conditional measures $\solm(dX \cond W)$---could be an artifact of the sensitive nature of stochastic integrals. However, in this additive driving noise setting, the property of being a solution to the underlying ODE is completely transparent and essentially classical.

As such, it becomes clear that weak uniqueness in the presence of pathwise nonuniqueness is a kind of very subtle \textit{uniqueness by measurable selection.} The closest well-understood problem which bears some analogies is that of ODE uniqueness in the DiPerna--Lions theory~\cite{diperna_ordinary_1989}. Specializing the discussion for the sake of simplicity, (extensions of) the DiPerna--Lions theory
proves that for divergence-free velocity fields $u \in C^0_t W^{1,1}_x$, there exists a unique \textit{regular Lagrangian flow}, which is an almost everywhere version of the ODE flow map obeying a certain ``compressibility bound''~\cite{ambrosio_transport_2004}. On the other hand, it is \textit{not} the case that---under the same hypotheses---there necessarily is uniqueness of the ODE for almost every initial condition~\cite{bruePositiveSolutionsTransport2021}. This situation is highly analogous to our own: we have a unique weak solution but don't have pathwise uniqueness, while in the DiPerna--Lions theory there is a unique regular Lagrangian flow but not almost everywhere uniqueness of the ODE. The analogy extends further, as unique existence of the regular Lagrangian flow was (originally) proved using PDE techniques and is strongly related to unique existence of bounded solutions to the associated transport equation for bounded data---while weak uniqueness for negative regularity flows is also proved with PDE techniques and is related to the well-posedness of the advection-diffusion equation~\cite{flandoli_multidimensional_2017}.

That said, the analogy between regular Lagrangian flows in DiPerna--Lions theory and weak solutions in SDE theory stops here. This is because for regular Lagrangian flows, while some measurability in the initial data is required, the constraint that actually is forcing the uniqueness of the flow is the compressibility constraint.\footnote{This is more-or-less clear from the fact that we can assume that all subsets of $\R^d$ are Lebesgue measurable while maintaining sufficiently much choice to leave the analysis unharmed by~\cite{solovay_model_1970}. More explicitly, one should be able to prove the existence of distinct measurable flows (that fail to obey the compressibility bound) under the positive measure ODE nonuniqueness statement of~\cite{bruePositiveSolutionsTransport2021}; some descriptive set theory may however be necessary to ensure the desired measurability.} For weak solutions, there is no constraint analogous to the compressibility bound; the only condition is that of a suitable measurability, as stated in Lemma~\ref{lem:disintegration}. However, due to the temporal structure, this measurability condition is much richer and concretely meaningful than a measurability condition on $\R^d$.\footnote{We note that a comparison has also been made between path-by-path uniqueness vs.\ pathwise uniqueness and almost everywhere ODE uniqueness vs.\ regular Lagrangian flow uniqueness in~\cite[p.\ 12]{flandoli_random_2011}. We believe that the comparison of (pathwise or path-by-path uniqueness) vs.\ weak uniqueness and almost everywhere ODE uniqueness vs.\ regular Lagrangian flow uniqueness is more apt. For an example that distinguishes pathwise from path-by-path uniqueness, see~\cite{shaposhnikov_pathwise_2022}.}  

Finally, we re-emphasize that all velocity fields considered in this work are divergence-free. Despite being divergence-free, these velocity fields are showing the sharpness of the pathwise uniqueness regime for generic velocity fields (with no constraint on the divergence, only constraining the regularity). This is in contrast to the weak well-posedness theory---as discussed in Section~\ref{s:previous results}---for which there are better results available under a divergence-free hypothesis than without any control on the divergence. This is sensible, since what we are concerned with is essentially an ODE nonuniqueness phenomenon---as Theorem~\ref{thm:main-general} and the heuristic above make clear---while a constraint on the divergence largely controls more ``global'' or ``coupled in initial data'' objects, such as the compression of the flow map or the change of $L^p$ norms of solutions to transport equations. Since the phenomenon we study is happening locally, for a fixed initial data, it is not too surprising that being divergence-free does not substantially alter the behavior.

\section{Construction of the velocity fields}

\label{sec:definition of fields}

In this section we construct the velocity fields $u^\rho$ and $v^\alpha$ appearing in Theorem~\ref{thm:main-general}. These are constructed out of a single shear flow at each time as this substantially simplifies the technical proof of the ``single-scale estimate'' appearing in Section~\ref{sec:one scale}, which formalizes a single step of the heuristic presented in Section~\ref{sec:heuristic-comp}.

We thus first introduce the family of random alternating shear flows that form the basis of $u^\rho$ and $v^\alpha$. Fix $\phi : \R \to [0,\infty)$ such that $\phi \in C_c^\infty((0,1))$, $\int \phi(t)\,dt =1$, and $|\phi|\leq 2$. Let $A^n$ be a sequence of iid uniform random variables on $[1/2,2]$ and $B^n$ a sequence of iid uniform random variables on $[0,2\pi].$ We then define the random velocity fields $V^i:[0,\infty)\times \R^2\rightarrow \R^2$ by
\begin{align}
    V^1(t,x)& := \sum_{j=0}^\infty \phi(t-j) \sin(A^j x_2+ B^j) e_1,\notag\\
    V^2(t,x)& := \sum_{j=0}^\infty\phi(t-j) \sin(A^j x_1+ B^j) e_2.
    \label{eq:V-def}
\end{align}
We then let $V^{n,1}, V^{n,2}$ denote independent sequences of random variables with laws equal to those of $V^1,V^2$ respectively. We thus have that $V^{n,i}$ shears in the direction $e_i$. The specific choice of these velocity fields in terms of the laws of $A^j$ and $B^j$ is made to simplify the proof of Proposition~\ref{prop:one scale} as much as possible, but is somewhat arbitrary.

Finally, to alternate between the shear flows, we define the alternating sequence 
\[\iota^n:=1+\tfrac{1}{2}(1+(-1)^n).\]
Thus $(\iota^1,\iota^2,\iota^3,\iota^4,\dotsc)=(1,2,1,2,\dotsc)$.

\subsection{\texorpdfstring{Construction of $u^\rho$ velocity fields}{Construction of velocity fields for Brownian driving noises}}\label{subsec:u-rho_const}

We first define the velocity fields $ u^\rho \in L^\infty_t C^{-\rho}_x \cap L^1_t C^0_x$ used in Theorem~\ref{thm:main-general}.

\begin{definition}\label{def:u-rho}
For $\rho\in(0,1/2)$ let $T^n$ be the sequence of times such that $\lim_{n\rightarrow\infty} T^n=0$ and
\[T^{n-1} - T^n = 2^{-(2+\rho/2)n} \lceil  2^{(2 - \frac{3}{4} \rho)n}\rceil.\]
Then we define the random vector field $u^\rho:[0,T^0]\times\R^2\rightarrow \R^2$ by
\[u^\rho(t,x) := \sum_{n=1}^\infty \indc_{t \in [T^n, T^{n-1})} 2^{\rho n} V^{n,\iota^n}(2^{(2+\rho/2)n} (t-T^n), 2^n x).\]
\end{definition}

Using the definition of the generalized H\"older norms with the explicit periodicity and scaling of the velocity fields defining $u^\rho$, we get the following regularity, which is proved by direct computation.

\begin{lemma}
    For all $\rho>0, k \in \N$, it surely holds that 
    \[u^\rho \in L^\infty([0,T^0], C^{-\rho}(\R^2)) \cap L^1([0,T^0],C^0(\R^2)) \cap W^{k,\infty}([0,T^0], C^{-\rho - (2+\frac{\rho}{2})k}(\R^2))\cap C^\infty_{\mathrm{loc}}((0,T^0]\times\R^2).\]
\end{lemma}

\begin{proof}
    We only prove that $u^\rho \in L^1_{t} C^0_x$ as it is the only estimate that does not follow by direct inspection. We then compute
    \[\|u^\rho\|_{L^1_tC^0_x} \leq \sum_{n=1}^\infty 2^{-(2+\rho/2)n} \lceil  2^{(2 - \frac{3}{4} \rho)n}\rceil 2^{\rho n}\|\phi\|_{L^\infty} \leq 2\|\phi\|_{L^\infty} \sum_{n=1}^\infty 2^{- \rho n/4} <\infty,\]
    allowing us to conclude.
\end{proof}

\subsection{\texorpdfstring{Construction of $v^\alpha$ velocity fields}{Construction of velocity fields for alpha-Holder driving noises}}\label{subseq:fBm_velocity}

Next, we define the velocity fields $v^\alpha \in L^\infty_t C^\alpha_x$ used in Theorem~\ref{thm:main-general}.

\begin{definition}\label{def:v-alpha}
For $\alpha\in\R$, let $T^n$ be the sequence of times such that $\lim_{n\rightarrow \infty}T^n=0$ and
\[T^{n-1} - T^n := 2^{-( 2 - 2\alpha - 2n^{-1/2})n} \lceil 2^{( 2 - 2\alpha - 3n^{-1/2}) n}\rceil.\]
We then define the random vector field $v^\alpha:[0,T^0]\times\R^2\rightarrow \R^2$ by
\[v^\alpha(t,x) := \sum_{n=1}^\infty \indc_{t \in [T^n, T^{n-1})} 2^{-\alpha n} V^{n,\iota^n}(2^{(2 - 2\alpha - 2 n^{-1/2}) n} (t-T^n), 2^n x).\]
\end{definition}

Again, by direct computation, we find $v^\alpha$ lives in the following regularity spaces. As this follows by direct inspection, we omit the argument.

\begin{lemma}
    For all $\alpha\in\R, k\in\N$, it surely holds that 
    \[v^\alpha \in L^\infty([0,T^0], C^{\alpha}(\R^2)) \cap W^{k,\infty}([0,T^0], C^{\alpha - (2-2\alpha)k}(\R^2)) \cap C^\infty_{\mathrm{loc}}((0,T^0]\times\R^2).\]
\end{lemma}

\section{Single scale separation estimate} \label{sec:one scale}

In this section we prove the ``single-scale estimate" which formalizes a single step of the heuristic presented in Section~\ref{sec:heuristic-comp}. Namely, a quantitative estimate on the particle separation generated from the independent hits from advection by our random velocity fields. It is here that we use that our velocity fields are constructed from shear flows. This substantially simplifies the argument since, for a shear pointing in the horizontal direction, two particles solving~\eqref{eq:main} with the same driving noise will not have their vertical separation change. Thus, since the change of their horizontal separation will only depend (in law) on their vertical separation (as well as on the driving noise), this will make it so the hits to the horizontal separation are truly independent. These hits will also be one dimensional, which is the simplest setting for the Berry--Esseen theorem.

One additional modification needed to make the argument presented in Section~\ref{sec:heuristic-comp} rigorous is that we need to keep track of the size of separations \textit{and} the probability they occur. That is, we cannot ensure with probability one that the separation is growing on each step in the desired way. Instead we prove that the separation grows in the desired way ``with high probability'', which in this setting means that the probability of failure is summable as $n \to \infty$. This summability will then imply the asymptotically almost sure statement given in Theorem~\ref{thm:main-general}.

First, we state the classical Berry--Esseen Theorem; see~\cite[Theorem 2.2.4]{stroock_probability_2024} for a modern treatment.
\begin{theorem}[Berry--Esseen Theorem]
\label{thm:Berry-esseen}
    Let $D_j$ be a sequence of mutually independent, $\R$-valued random variables with $\E [D_j] =0$ and let $Z$ be a standard normal random variable on $\R$. Let
    \[\bar S_n := \Big(\sum_{j=1}^n \E [D_j^2]\Big)^{-1/2}\sum_{j=1}^n D_j.\]    
    Then for all $n \in \N$ and $x \in \R$, 
    \[\big|\P(\bar S_n \leq x) - \P(Z \leq x)\big| \leq 10 \frac{\sum_{j=1}^n \E[|D_j|^3]}{\big(\sum_{j=1}^n \E[D_j^2]\big)^{3/2}}.\]
\end{theorem}

We now state and prove the ``single-scale estimate." This quantifies the probability that under shearing by $V^1$, particles that are initially separated in the vertical direction separate in the horizontal direction for a fixed driving noise $W$. This implies the analogous result for $V^2$ after rotation. Throughout the rest of the section, given some element $x\in\R^2$, we let $x=(x_1,x_2)$ denote the coordinates.

\begin{proposition}
\label{prop:one scale}
    Fix $W \in C^0([0,\infty),\R^2)$ and let $u=V^1$ defined by~\eqref{eq:V-def}. For $N\in\N$, additionally suppose that 
    \begin{equation}\label{eq:regularity_cond}
    \sup_{n<N} \sup_{s,t \in [0,1]} |W_{s+n,2} - W_{t+n,2}| \leq \tfrac{1}{16}.
    \end{equation}
    Then for all $y^1,y^2\in\R^2$ such that $|y^1_2-y^2_2| \geq 4$, the unique solutions $X_t^1, X_t^2$ to~\eqref{eq:main} with $u = V^1$ and $(X_0^1,X_0^2)=(y^1,y^2)$ have the property that
    \[\volm(|X_{N,1}^1 - X_{N,1}^2| \leq L) \leq (2L + 640)N^{-1/2},\quad \forall L>0.\]
\end{proposition}

\begin{proof}
    For simplicity, let $R_t :=X_{t,1}^1 - X_{t,1}^2$ denote the horizontal separation between $X^1_t$ and $X^2_t$, and $a=X_{0,2}^1,b=X_{0,2}^2$ their initial vertical coordinates. We will suppose without loss of generality that $W_0=0$.

    By the definition of $V^1$, we thus have that in law
    \begin{align*}
        X_t^1 &= y^1 + W_t + \int_0^t \sum_{j=0}^\infty \phi(s - j) \sin(A^j (a+W_{s,2})+ B^j) e_1\,ds,\\
         X_t^2 &= y^2 + W_t + \int_0^t \sum_{j=0}^\infty \phi(s-j) \sin(A^j (b+W_{s,2})+ B^j) e_1 \,ds,
    \end{align*}
    where we have used that $V^1$ has no $e_2$ component. We thus find that
    \[D_n:=R_{n+1}-R_n=\int_n^{n+1} \phi(s-n) \big(\sin(A^n (a+W_{s,2})+ B^n) - \sin(A^n (b+W_{s,2})+ B^n)\big)\,ds,\]
    where the last equality is in law.

    Since the $A^n$ and $B^n$ are independent across $n$, so are the increments $D_n$, although they are not necessarily identically distributed due to the $W_{s,2}$ term. Additionally, since $B^n$ is uniform $[0,2\pi]$ and $\sin(x)$ is mean-zero, $D_n$ has zero mean for all $n$ since
    \[\E [D_n] =\frac{1}{3\pi}\int_{1/2}^2\int_n^{n+1} \phi(s-n)\int_0^{2\pi} \big(\sin(y (a+  W_{s,2})+z) - \sin(y (b+W_{s,2})+ z)\big)\,dzdsdy= 0.\]
    We next want to derive uniform lower bounds on $\E [D_n^2]$ and uniform upper bounds on $\E [|D_n|^3]$ with the goal of applying the Berry--Esseen theorem, Theorem~\ref{thm:Berry-esseen} to $\sum_{n=0}^{N-1} D_n$.

    To this end, for $n$ fixed let $w_t := W_{t+n,2} - W_{n,2}$. We then note that, using the periodicity of $\sin$ and the distributions of $A^n$ and $B^n$
    \begin{align*}\E [D_n^2] &=  \frac{1}{3\pi}\int_0^{2\pi}\int_{1/2}^2\bigg(\int_0^1\phi(s)(\sin(y(a+W_{s,2})+z)-\sin(y(b+W_{s,2})+z)\,ds\bigg)^2\,dy\,dz
    \\&=  \frac{1}{3\pi}\int_0^{2\pi}\int_{1/2}^2\bigg(\int_0^1\phi(s)(\sin(y(a-b+w_s)+z)-\sin(yw_s+z)\,ds\bigg)^2\,dy\,dz.
    \end{align*}
    Expanding out the square and integrating we thus find that
    \begin{align*}
    \E [D_n^2] &= \frac{1}{3\pi}\int_0^{1} \phi(r) \int_0^{1} \phi(s)\int_{1/2}^2 \int_0^{2\pi}   \sin(y(a-b+w_s)+ z)\sin(y(a-b+w_r)+ z) 
    \\&\qquad\qquad\qquad + \sin(yw_s+z) \sin(yw_r+z)- 2\sin(y(a-b+w_s)+ z) \sin(y w_r+z)\,dzdydsdr
    \\&=\frac{2}{3} \int_0^{1} \phi(r) \int_0^{1}  \phi(s)\int_{1/2}^2  \cos(y(w_s - w_r)) -  \cos(y(a-b+w_s-w_r))\, dydsdr.
    \end{align*}
    Since $\cos(x)\geq 1-\frac{1}{2}|x|^2$ for all $x$ and $y\leq 2$,
    \[\frac{2}{3} \int_0^{1} \phi(r) \int_0^{1}  \phi(s)\int_{1/2}^2  \cos(y(w_s - w_r))\,dydsdr\geq 1-2\sup_{s,t\in[0,1]}|w_s-w_r|^2.\]
    On the other hand, integrating again
    \begin{align*}\frac{2}{3} \int_0^{1} \phi(r) \int_0^{1}  &\phi(s)\int_{1/2}^2 \cos(y(a-b + w_s-w_r))\, dydsdr
    \\&\qquad=-\frac{2}{3}\int_0^1\int_0^1\bigg[\frac{1}{a-b+w_s-w_r}\sin(y(a-b+w_s-w_r))\bigg]^{y=2}_{y=1/2}\phi(r)\phi(s)\,drds
    \\&\qquad\leq\frac{4}{3}\big(|a-b|-\sup_{s,r\in[0,1]}|w_s-w_r|\big)^{-1}.
    \end{align*}
    Combining the above three math displays, we have thus found that
    \[\mathbb{E}[D_n^2]\geq 1-2\sup_{s,r\in[0,1]}|w_s-w_r|^2-\frac{4}{3}\big(|a-b|-\sup_{s,r\in[0,1]}|w_s-w_r|\big)^{-1}\geq \frac{1}{2},\]
    where in the last inequality we use the fluctuation assumption~\eqref{eq:regularity_cond} and the lower bound on $|a-b|$. On the other hand, we easily see that $\E [|D_n|^3]\leq 8$.

    We are now ready to apply Theorem~\ref{thm:Berry-esseen}. Let $r_0 = X_{0,1}^1 - X_{0,1}^2$, $\Sigma_N := \sum_{n=0}^{N-1} \E[D_n^2]$, and $\bar S_N := \Sigma_N^{-1/2} \sum_{n=0}^{N-1} D_n$ so that $R_N=r_0+\Sigma_N^{1/2}\bar S_N$. Then we note that
    \begin{align*}
    \volm(|X_{N,1}^1 - X_{N,1}^2| \leq L)& = \volm(R_N \in [-L, L]) 
    \\&= \volm\Big(\bar S_N \in [\Sigma_N^{-1/2}(-L - r_0), \Sigma_N^{-1/2}(L-r_0)]\Big)
    \\&= \volm\big(\bar S_N \leq  \Sigma_N^{-1/2}(L-r_0)\big) - \volm\big(\bar S_N \leq \Sigma_N^{-1/2}(-L - r_0)\big).
    \end{align*}
    Next we apply Theorem~\ref{thm:Berry-esseen}, noting that by our above computations
    \[\frac{\sum_{n=0}^{N-1} \E[|D_n|^3]}{\big(\sum_{n=0}^{N-1} \E[D_n^2]\big)^{3/2}} \leq \frac{8N}{(2^{-1} N)^{3/2}} \leq 32 N^{-1/2},\]
    giving together that, for a standard normal random variable $Z$ on $\R$,
    \begin{align*}\volm(|X_{N,1}^1 - X_{N,1}^2| \leq L) &\leq \P\big(Z \in  [\Sigma_N^{-1/2}(-L - r_0), \Sigma_N^{-1/2}(L-r_0)]\big) + 640 N^{-1/2} 
    \\&\leq 2(2\pi)^{-1/2} \Sigma_N^{-1/2} L + 640 N^{-1/2},
    \end{align*}
    where for the second inequality we use the $L^\infty$ bound on the density of the standard normal. Then we note that $\Sigma_N \geq N/2$, so in total we find
    \[\volm(|X_{N,1}^1 - X_{N,1}^2| \leq L) \leq  (2(2\pi)^{-1/2} 2^{1/2} L + 640 )N^{-1/2} \leq (2L + 640)N^{-1/2},\]
    as claimed.
\end{proof}

\section{Multiscale iteration}\label{sec:multiscale}

In this section we prove Theorem~\ref{thm:main-general} by establishing quantitative separation estimates for $u^\rho$ and $v^\alpha$. These are given in Propositions~\ref{prop:brownian separation} and~\ref{prop:separation_estimate}, from which the theorem follows directly.

The starting point is to rescale Proposition~\ref{prop:one scale} in time and space so that it applies on each interval $[T^n,T^{n-1}]$ defining $u^\rho$ and $v^\alpha$. This yields the one-step separation estimates in Corollaries~\ref{cor:rescaled for brownian} and~\ref{cor:rescaled for general}. It is at this stage that we use the regularity of the driving noise as it ensures that, after rescaling to sufficiently small time intervals, condition~\eqref{eq:regularity_cond} is satisfied.

These one-step estimates are then iterated across scales, yielding Propositions~\ref{prop:brownian separation} and~\ref{prop:separation_estimate}. Concretely, if $m>n$ are sufficiently large and the particles are separated at time $T^m$ at the correct scale and in the correct direction, then with high probability this separation propagates to time $T^n$ at the corresponding scale. The mechanism is as follows. Suppose $\iota^m=1$, so that the velocity field acts horizontally on $[T^m,T^{m-1}]$, and the particles have vertical separation at least $2^{-m+2}$ at time $T^m$. The one-step estimate implies that, with high probability, their horizontal separation at time $T^{m-1}$ is at least $2^{-m+3}$. Conditioning and iterating this argument while alternating between vertical and horizontal directions yields separation at time $T^n$ of size at least $2^{-n+2}$ with high probability. The case $\iota^m=2$ then follows by symmetry.

\subsection{Rescaled single scale estimates}

We begin with the rescaled one-step separation estimates. For this purpose, recall that, as defined in Definition~\ref{def:sol}, $X_t^{s,y,u,W}=(X^{s,y,u,W}_{t,1},X^{s,y,u,W}_{t,2})$ denotes the classical solution to an ODE with initial condition $y$ at time $s$, advecting velocity field $u\in C^\infty_{\mathrm{loc}}((0,T]\times\R^2)$, and driving noise $W$.

We first state the estimate for $u^\rho$. Note that $n$ must be sufficiently large relative to the $C^{\frac{1+\rho/8}{2+\rho/2}}$ norm of the driving noise.

\begin{corollary}\label{cor:rescaled for brownian}
    Let $\rho \in (0,1/2)$, $T^j$ and $u^\rho$ as given by Definition~\ref{def:u-rho}, and $W\in C^{\frac{1+\rho/8}{2+\rho/2}}([0,T^0],\R^2)$.
    Then for all $y^1,y^2\in\R^2$ and $n\in\N$ such that 
    \[|y^1_{\iota^{n+1}}-y^2_{\iota^{n+1}}|\geq 2^{-n+2},\] 
    and
    \[n \geq 32 \rho^{-1} + 8 \rho^{-1} \log_2\big(\|W\|_{C^{\frac{1+\rho/8}{2+\rho/2}}([0,T^0],\R^2)}\big).\]
    it holds that
    \[\volm\big(|X^{T^n,y^1,u^\rho,W}_{T^{n-1},\iota^{n}} - X^{T^n,y^2,u^\rho,W}_{T^{n-1},\iota^n}| \leq 2^{-n+3}\big) \leq 64 \cdot 2^{-\rho n/8}.\]
\end{corollary}

\begin{proof}
    We proceed by rescaling our solutions. Assume without loss of generality that $\iota^n=1$, and let
    \[\tilde X_t^i := \Big( 2^{(2-\rho/2)n} X^{T^n,y^i,u^\rho,W}_{T^n+2^{-(2+\rho/2)n}t,1}, 2^nX^{T^n,y^i,u^\rho,W}_{T^n+2^{-(2+\rho/2)n} t,2}\Big),\qquad i\in\{1,2\},\]
    and
    \[\tilde W_{t}:=\Big(2^{(2-\rho/2)n}W_{T^n+2^{-(2+\rho/2)n}t,1}, 2^n W_{T^n+2^{-(2+\rho/2)n} t,2}\Big).\]
    Then $\tilde{X}^i$ solves
    \[\begin{cases}
        d \tilde X_t^i =  V^1( t, \tilde X_t^i)dt + d\tilde W_t,
        \\\tilde{X}_0^i=(2^{(2-\rho/2)n}y^i_1,2^ny^i_2),
    \end{cases}\]
    and, letting $N:= \lceil 2^{( 2 - \frac{3}{4} \rho) n} \rceil$, we have $\tilde{X}^i_{N,1}=2^{(2-\rho/2)n}X^{T^n,y^i,u^\rho,W}_{T^{n-1},1}$. We also compute that
     \[ \sup_{k< N}\sup_{s,t \in [0,1]} |\tilde W_{s+k,2} - \tilde W_{t+k,2}| \leq  2^{-\rho n/8} \|W\|_{C^{\frac{1+\rho/8}{2+\rho/2}}} \leq \tfrac{1}{16},\]
     where we use our assumption on $n$ for the final inequality. 

    By our assumptions on $y^i$, $|\tilde X_{0,2}^1 - \tilde X_{0,2}^2| \geq 4$, so we are exactly in the setting of Proposition~\ref{prop:one scale}. Letting $\tilde\volm$ be the law of $V^1$, Proposition~\ref{prop:one scale} gives that
    \begin{align*}\volm\big(|X^{T^n,y^1,u^\rho,W}_{T^{n-1},1} - X^{T^n,y^2,u^\rho,W}_{T^{n-1},1}| \leq 2^{-n+3}\big)  &= \tilde\volm\big( |\tilde X_{N,1}^1 - \tilde X_{N,1}^2| \leq 2^{3+(1 - \rho/2)n}\big)
    \\&\leq ( 2^{4+(1- \rho/2)n}+ 640)N^{-1/2}.
        \\&\leq 64\cdot2^{-\rho n /8}
    \end{align*}
    where in the last inequality we have used that $2^{4+(1- \rho/2)n}\geq 640$ since $\rho\in(0,1/2)$ and $n\geq 32\rho^{-1}$. This concludes the claim.
\end{proof}

Next, we state and prove the rescaled estimate for $v^\alpha$. This follows almost verbatim as Corollary~\ref{cor:rescaled for brownian}.

\begin{corollary}
\label{cor:rescaled for general}
    Let $\alpha \in \R$ with $\alpha < 1/2$, $T^j$ and $v^\alpha$ as given by Definition~\ref{def:v-alpha}, $\frac{1}{2(1-\alpha)} < \beta \leq 1,$ and $W \in C^\beta([0,T^0],\R^2)$.
    Then for all $y^1,y^2\in\R^2$ and $n\in\N$ such that 
    \[|y^1_{\iota^{n+1}}-y^2_{\iota^{n+1}}|\geq 2^{-n+2},\]
    and
    \[n \geq \Big(\frac{4\beta}{2\beta(1-\alpha) - 1}\Big)^2 \lor \frac{8 + 2\log_2 \|W\|_{C^\beta([0,T^0],\R^2)} }{2\beta(1-\alpha)-1} \lor 64,\]
    it holds that
    \[\volm(|X^{T^n,y^1,v^\alpha,W}_{T^{n-1},\iota^n} - X^{T^n,y^2,v^\alpha,W}_{T^{n-1},\iota^n}| \leq 2^{-n+3}) \leq 64\cdot 2^{-n^{1/2}/2}.\]
\end{corollary}

\begin{proof}
    We again assume without loss of generality that $\iota^n=1$, but now let
    \[\tilde X_t^i=\Big(2^{(2 - \alpha - 2n^{-1/2}) n} X^{T^n,y^i,v^\alpha,W}_{T^n+2^{-(2 - 2\alpha - 2 n^{-1/2}) n}t,1}, 2^nX^{T^n,y^i,v^\alpha,W}_{T^n+2^{-(2 - 2\alpha - 2 n^{-1/2}) n}t,2}\Big),\qquad i\in\{1,2\},\]
    and
    \[\tilde W_t =\Big(2^{(2 - \alpha - 2n^{-1/2}) n}W_{T^n+2^{-(2-2\alpha - 2 n^{-1/2})n} t,1} ,2^n W_{T^n+2^{-(2-2\alpha - 2 n^{-1/2})n} t,2}\Big).\]
    Then $\tilde{X}^i$ solves
    \[\begin{cases}
        d \tilde X_t^i =  V^1( t, \tilde X_t^i)dt + d\tilde W_t,
        \\\tilde{X}_0^i=(2^{(2 - \alpha - 2n^{-1/2}) n}y^i_1,2^ny^i_2),
    \end{cases}\]
    and, letting $N=\lceil 2^{( 2 - 2\alpha - 3n^{-1/2}) n} \rceil$, we have $\tilde{X}^i_{N,1}=2^{(2 - \alpha - 2n^{-1/2}) n}X^{T^n,y^i,v^\alpha,W}_{T^{n-1},1}$. We also compute that
     \[ \sup_{k<N}\sup_{s,t \in [0,1]} |\tilde W_{s+k,2} - \tilde W_{t+k,2}| \leq 2^{(1-\beta (2-2\alpha -2n^{-1/2})) n} \|W\|_{C^\beta} \leq \tfrac{1}{16},\]
    where we use our assumption on $n$ for the final inequality.

    By our assumptions on $y^i$,  $|\tilde X_{0,2}^1 - \tilde X_{0,2}^2| \geq 4$, so we are again exactly in the setting of Proposition~\ref{prop:one scale}. Letting $\tilde{U}$ be the law of $V^1$, Proposition~\ref{prop:one scale} gives that
    \begin{align*}\volm(|X^{T^n,y^1,v^\alpha,W}_{T^{n-1},1} - X^{T^n,y^2,v^\alpha,W}_{T^{n-1},1}| \leq 2^{-n+3})  &= \tilde \volm\big( |\tilde X_{N,1}^1 - \tilde X_{N,1}^2| \leq 2^{3+(1- \alpha - 2n^{-1/2})n}\big) 
    \\&\leq ( 2^{4+(1- \alpha - 2n^{-1/2})n}+ 640)N^{-1/2}
    \\&\leq 64\cdot 2^{-n^{1/2}/2},
    \end{align*}
    where in the last inequality we have used that $2^{4+(1- \alpha - 2n^{-1/2})n}\geq 640$ since $n\geq 64$. This concludes the claim.
\end{proof}

\subsection{Multiscale estimates}

We can now iterate Corollary~\ref{cor:rescaled for brownian} and Corollary~\ref{cor:rescaled for general} to get the multi-scale quantitative separation estimates for $u^\rho$ and $v^\alpha$. We begin with the multi-scale estimate for $u^\rho$.

\begin{proposition}
    \label{prop:brownian separation}
     Let $\rho \in (0,1/2)$, $T^j$ and $u^\rho$ as given by Definition~\ref{def:u-rho}, and $W \in C^{\frac{1+\rho/8}{2+\rho/2}}([0,T^0],\R^2)$. Then for all $y^1,y^2 \in \R^2$ and $m,n\in\N$  such that
     \[|y^1_{\iota^{m+1}}-y^2_{\iota^{m+1}}| \geq 2^{-m+2},\]
     and
     \[ m\geq n\geq 32 \rho^{-1} + 8 \rho^{-1} \log_2\big(\|W\|_{C^{\frac{1+\rho/8}{2+\rho/2}}([0,T^0],\R^2)}\big),\]
     it holds that
     \[\volm\Big(|X^{T^m,y^1,u^\rho,W}_{T^n} - X^{T^m,y^2,u^\rho,W}_{T^n}| \leq 2^{-n+2}\Big) \leq  768 \rho^{-1} 2^{-\rho n/8}.\]
\end{proposition}

\begin{proof}
    For simplicity, we drop the $u^\rho$ and $W$ in the superscripts of the solutions so that 
    \[X^{T^m,y^i}_{T^n}:=X^{T^m,y^i,u^\rho,W}_{T^n}.\]
    Iteratively conditioning, we then write that
    \begin{align}
        &\notag \volm\big(|X^{T^m,y^1}_{T^n} - X^{T^m,y^2}_{T^n}| \leq 2^{-n+2}\big) 
        \\&\qquad\qquad\notag \leq \volm\big(|X^{T^m,y^1}_{T^n,\iota^{n+1}} - X^{T^m,y^2}_{T^n,\iota^{n+1}}| \leq 2^{-n+2} \,\big|\, |X^{T^m,y^1}_{T^{n+1},\iota^{n+2}} - X^{T^m,y^2}_{T^{n+1},\iota^{n+2}}| > 2^{-n+1}\big)
        \\&\notag\qquad\qquad\qquad\qquad+ \volm\big(|X^{T^m,y^1}_{T^{n+1},\iota^{n+2}} - X^{T^m,y^2}_{T^{n+1},\iota^{n+2}}| \leq 2^{-n+1}\big)
        \\&\notag\qquad\qquad\leq \sum_{j=n}^{m-1} \volm\big(|X^{T^m,y^1}_{T^j,\iota^{j+1}} - X^{T^m,y^2}_{T^j,\iota^{j+1}}| \leq 2^{-j+2} \,\big|\, |X^{T^m,y^1}_{T^{j+1},\iota^{j+2}} - X^{T^m,y^2}_{T^{j+1},\iota^{j+2}}| > 2^{-j+1}\big)
        \\&\label{eq:iterated_condition}\qquad\qquad\qquad\qquad + \volm\big(|X^{T^m,y^1}_{T^{m},\iota^{m+1}} - X^{T^m,y^2}_{T^{m},\iota^{m+1}}| \leq 2^{-m+2}\big).
    \end{align}
    By our hypothesis on $m$, $\volm(|X^{T^m,y^1}_{T^{m},\iota^{m+1}} - X^{T^m,y^2}_{T^{m},\iota^{m+1}}| \leq 2^{-m+2}) = \volm(|y^1_{\iota^{m+1}}-y^2_{\iota^{m+1}}| < 2^{-m+2}) = 0$. On the other hand, we note that 
    \[X^{T^m,y^i}_{T^j}=X^{T^{j+1},X^{T^m,y^i}_{T^{j+1}}}_{T^j}.\]
    Since $u^\rho|_{[0,T^{j+1}]}$ is independent of $u^\rho|_{(T^{j+1},T^j]}$, Corollary~\ref{cor:rescaled for brownian} and our hypothesis on $n$ thus imply that for all $n\leq j\leq m-1,$
    \[\volm\big(|X^{T^m,y^1}_{T^j,\iota^{j+1}} - X^{T^m,y^2}_{T^j,\iota^{j+1}}| \leq 2^{-j+2}\,\big|\,|X^{T^m,y^1}_{T^{j+1},\iota^{j+2}} - X^{T^m,y^2}_{T^{j+1},\iota^{j+2}}| > 2^{-j+1}\big) \leq 64 \cdot 2^{- \rho(j+1)/8}.\]
    Inserting this into~\eqref{eq:iterated_condition} and bounding the sum, we get the claimed result.
\end{proof}

Next, we state the multi-scale estimate for $v^\alpha$. As the proof follows exactly as that of Proposition~\ref{prop:brownian separation} but with Corollary~\ref{cor:rescaled for general} in place of Corollary~\ref{cor:rescaled for brownian}, we omit the argument.

\begin{proposition}\label{prop:separation_estimate}
    Let $\alpha\in\R$ with $\alpha< 1/2$, $T^j$ and $v^\alpha$ as given by Definition~\ref{def:v-alpha}, $\frac{1}{2(1-\alpha)} < \beta \leq 1$, and $W \in C^\beta([0,T^0],\R^2)$. Then for all $y^1,y^2\in\R^2$ and $m,n\in\N$ such that 
    \[|y^1_{\iota^{m+1}}-y^2_{\iota^{m+1}}| \geq 2^{-m+2},\]
    and
     \[m\geq n\geq \Big(\frac{4\beta}{2\beta(1-\alpha) - 1}\Big)^2 \lor \frac{8 + 2\log_2 \|W\|_{C^\beta([0,T^0],\R^2)} }{2\beta(1-\alpha)-1} \lor 64,\]
     it holds that 
     \[\volm\Big(|X^{T^m,y^1,v^\alpha,W}_{T^n} - X^{T^m,y^2,v^\alpha,W}_{T^n}| \leq 2^{-n+2}\Big) \leq 512\cdot 2^{-\sqrt{n}/4}.\]
\end{proposition}

Finally, we note that Theorem~\ref{thm:main-general} follows immediately from Proposition~\ref{prop:brownian separation} and Proposition~\ref{prop:separation_estimate} by taking consecutive limits (importantly taking them in the correct order).

\section{Explosive separation to qualitative nonuniqueness}\label{sec:quantitative to qualitative}

In this section our goal is to prove Theorem~\ref{thm:general_non_path}, which allows us to take the explosive separation of Theorem~\ref{thm:main-general} to prove the nonuniqueness results of Theorem~\ref{thm:main-brownian}, Corollary~\ref{cor:caratheodory general}, and Corollary~\ref{cor:pathwise failure fBm}. It may be useful to recall the notational Remark~\ref{rem:measure notation}, as we will be working with the probability measures $\volm, \W, \solm,$ and $\coupm$ in this section.

The first lemma we want to prove is essentially a version of Fubini's theorem. We assume that $\W$-a.s.\ we have explosive separation, defined in Definition~\ref{def:explosive separation}, for the $\volm$ measure. We turn this statement about $\W$-a.s.\ limits of $\volm$ probabilities into a statement about $\volm$-a.s.\ separation in asymptotically high $\W$ probability. That is, we get a statement which is a $\volm$-a.s.\ limit of $\W$ probabilities. However, as these a.s.\ limits don't directly ``commute'' with the measures in this way, we are forced to replace the $\limsup$'s with $\liminf$'s. This constitutes the essential separation ingredient to the proof of nonuniqueness in Theorem~\ref{thm:general_non_path}.

\begin{lemma}\label{lem:fubini}
Suppose that for a driving noise with measure $\mathbb{W}\in \mathcal{P}\big(C^0([0,1],\R^d)\big)$, a random velocity field with measure $\volm \in \mathcal{P}\big(C^\infty_{\mathrm{loc}}((0,1] \times \R^d)\big)$, and a sequence of times $T^n\rightarrow 0^+$, $u$ is explosively separating for $T^n$ and $\mathbb{W}$-almost every $W$. Then for any $y\in\R^d$
\[\liminf_{n\rightarrow \infty}\liminf_{\delta\rightarrow 0} \liminf_{x\rightarrow y}\liminf_{m\rightarrow \infty}\W(|X^{T^m,x,u, W}_{T^n} - X^{T^m,y,u, W}_{T^n}|<\delta)=0,\qquad \volm\text{-a.e.\;}u.\]
\end{lemma}

\begin{proof}
Let $\P=\volm\otimes \mathbb{W}$ denote the joint law of $(u,W)$. Then we have that
\[\P(|X^{T^m,x,u, W}_{T^n} - X^{T^m,y,u, W}_{T^n}|<\delta)=\mathbb{E}_{\mathbb{W}}\big[\volm(|X^{T^m,x,u, W}_{T^n} - X^{T^m,y,u, W}_{T^n}|<\delta)\big],\]
where $\E_\W$ denotes integration over the $\W$ measure (similarly $\E_{\volm}$).
Since $u$ is explosively separating $\mathbb{W}$-a.s., the (reverse) Fatou lemma, gives that
\[ \lim_{n \to \infty} \limsup_{\delta \to 0} \limsup_{|x-y| \to 0} \limsup_{m \to \infty} \P(|X^{T^m,x,u, W}_{T^n} - X^{T^m,y,u, W}_{T^n}|<\delta)=0.\]
We then note that for any $\ep>0,$
\begin{align*}\P(|X^{T^m,x,u, W}_{T^n} - X^{T^m,y,u, W}_{T^n}|<\delta)&=\mathbb{E}_\volm\big[\mathbb{W}(|X^{T^m,x,u, W}_{T^n} - X^{T^m,y,u, W}_{T^n}|<\delta)\big]
\\&\geq  \eps\volm\big(\mathbb{W}(|X^{T^m,x,u, W}_{T^n} - X^{T^m,y,u, W}_{T^n}|<\delta)>\eps\big).
\end{align*}
Thus we have for all $\ep>0$,
\begin{equation}\label{eq:annealed}
\lim_{n \to \infty} \limsup_{\delta \to 0} \limsup_{|x-y| \to 0} \limsup_{m \to \infty} \volm\big(\W(|X^{T^m,x,u, W}_{T^n} - X^{T^m,y,u, W}_{T^n}|<\delta)>\ep\big) =0.
\end{equation}

Fixing $y\in\R^d$ and letting
\[Z^{n,\delta,x,m}:=\W(|X^{T^m,x,u, W}_{T^n} - X^{T^m,y,u, W}_{T^n}|<\delta),\]
equation~\eqref{eq:annealed} immediately implies that for all $\ep>0$,
\[\lim_{n\rightarrow \infty}\limsup_{\delta\rightarrow 0} \limsup_{x\rightarrow y}\limsup_{m\rightarrow \infty}\volm\big(Z^{n,\delta,x,m}>\eps\big)=0.\]
Fatou's lemma thus implies that
\begin{align*}
    0&= \lim_{n\rightarrow \infty}\limsup_{\delta\rightarrow 0} \limsup_{x\rightarrow y}\limsup_{m\rightarrow \infty}\mathbb{E}_{\volm}\Big [\indc_{\{Z^{n,\delta,x,m}>\eps\}}\Big]
    \\& \geq \mathbb{E}_{\volm}\Big[  \liminf_{n\rightarrow \infty}\liminf_{\delta\rightarrow 0} \liminf_{x\rightarrow y}\liminf_{m\rightarrow \infty}\indc_{\{Z^{n,\delta,x,m}>\eps\}}\Big]
    \\&\geq \mathbb{E}_{\volm}\Big[\indc_{\{ \liminf_{n\rightarrow \infty}\liminf_{\delta\rightarrow 0} \liminf_{x\rightarrow y}\liminf_{m\rightarrow \infty} Z^{n,\delta,x,m}>\eps\}}\Big]
    \\&= \volm\Big(\liminf_{n\rightarrow \infty}\liminf _{\delta\rightarrow 0} \liminf_{x\rightarrow y}\liminf_{m\rightarrow \infty} Z^{n,\delta,x,m}>\eps\Big).
\end{align*}
Unpacking the definition of $Z^{n,\delta,x,m}$, we conclude.
\end{proof}

We now state the lemma giving the compactness of weak solutions. We note here that we take $u$ to be a fixed, deterministic velocity field. It is in this lemma (and only this lemma) that we use regularity of the driving noise to preserve the independence condition Item~\ref{item:non-anticipatory} of Definition~\ref{def:weak} in the weak limit of measures.

\begin{lemma}\label{lemma:compactness} Let $u\in L^1([0,T],C^0(\R^d))$ be a deterministic velocity field, $\W\in \mathcal{P}(C^0([0,T],\R^d))$ a regular driving noise, $y^n,y\in\R^d$, and $\tau^n\geq 0$ such that $(y^n,\tau^n)\rightarrow (y,0)$ as $n\rightarrow \infty$. Further, for $n \in \N$, let $X^n$ be a weak solution to
\begin{equation}\label{eq:SDE_sequence}
\begin{cases}
    dX_t^n=u(t,X_t^n)\,dt+dW_t,\\
    X_{\tau^n}^n=y^n,
\end{cases}
\end{equation}
with driving noise $\W$. We naturally extend $X_t^n$ to a process on $C^0([0,T],\R^d)$ by letting $X_t^n=y^n$ for all $t\in[0,\tau^n]$.

Then there exists a weak solution to 
\begin{equation}\label{eq:SDE_limit}
\begin{cases}
    dX_t=u(t,X_t)\,dt+ dW_t\\
    X_0=y
\end{cases}
\end{equation}
with driving noise $\W$ such that, up to a subsequence, $(X^n,W)$ converges to $(X,W)$ weakly in law on $C^0([0,T],\R^d) \times C^0([0,T],\R^d)$.
\end{lemma}

We defer the proof of the above lemma---as well as the following---to Appendix~\ref{appen:weak_sol_theory}. The final lemma we will need is a very abstract result allowing us to construct couplings that charge no mass to the diagonal. This is necessary for the proof of Item~\ref{item:weak unique as pathwise nonunique} of Theorem~\ref{thm:general_non_path}.

\begin{lemma}
\label{lem:coupling no diagonal}
    Let $\mathcal{X},\mathcal{Y}$ be Polish spaces, $\mathcal{Y}$ uncountable, and $x \mapsto \mu_x$ a (Borel) measurable map from $\mathcal{X} \to \mathcal{P}(\mathcal{Y}).$ Suppose that $\nu \in \mathcal{P}(\mathcal{X})$ and $\nu$-a.s., $\mu_x$ does not have an atom with mass greater than $1/2.$ Then there exists a (Borel) measurable map $x \mapsto \gamma_x$ from $\mathcal{X} \to \mathcal{P}(\mathcal{Y} \times \mathcal{Y})$ such that $\nu$-a.s., $\gamma_x(dy, \mathcal{Y}) = \gamma_x(\mathcal{Y}, dy) = \mu_x$ and $\gamma_x(\{(y,y) : y \in \mathcal{Y}\}) =0.$
\end{lemma}

With Lemmas~\ref{lem:fubini}-\ref{lem:coupling no diagonal} in hand, we are ready to prove Theorem~\ref{thm:general_non_path}.

\begin{proof}[Proof of Theorem~\ref{thm:general_non_path}]
    Fix $y\in\R^d$. By our hypothesis and Lemma~\ref{lem:fubini}, we have that 
    \[\liminf_{n\rightarrow \infty}\liminf_{\delta \rightarrow 0} \liminf_{x\rightarrow y}\liminf_{m\rightarrow \infty}\W(|X^{T^m,x,u, W}_{T^n} - X^{T^m,y,u, W}_{T^n}|<\delta)=0,\qquad \volm\text{-a.e.\;}u.\]
    Fix $u$ in the full probability set for which this holds.
    
    We first prove Item~\ref{item:high prob pathwise}. Fix $\ep,\gamma \in (0,1)$. Then, by the above display, we can find ($u,\gamma$-dependent) $T \in [0,\ep],\delta>0$, and sequences $\tau^\ell \in (0,T]$ and $y^\ell \in \R^d$, such that $\tau^\ell \to 0; y^\ell \to y;$ and
    \[\forall \ell \in \N,\, \W(|X^{\tau^\ell,y^\ell,u, W}_{T} - X^{\tau^\ell,y,u, W}_{T}|< \delta)\leq \gamma.\]
    Then let 
    \[X^{1,\ell}_t:=X^{\tau^\ell,y^\ell,u, W}_{t},\quad X^{2,\ell}_t:=X^{\tau^\ell,y,u, W}_{t},\quad V_t:=(W_t,W_t)^{\top},\quad\text{and}\quad  Y_t^\ell:=(X_t^{1,\ell},X^{2,\ell}_t)^\top.\]
    Letting $\mathbb{V}$ denote the law of $V$, we note that $\mathbb{V}(dV^1,dV^2)= \int_{C^0([0,1])} \delta_{W}(dV^1)\delta_{W}(dV^2)\mathbb{W}(dW)$. Since $\mathbb{W}$ is a regular driving noise, so is $\mathbb{V}$. Then for all $\ell$, the tuple $(Y^\ell,V)$ defines a weak solution---noting here that we are taking $u$ fixed and the weak solution is (only) with respect to $\mathbb{V}$---to the SDE
    \begin{equation}
    \begin{cases}
        dY_t^\ell=U(t,Y_t^\ell)\,dt+dV_t,\\
        Y_{\tau^\ell}^\ell=(y^\ell,y)^\top,
    \end{cases}
    \end{equation}
    where $U(t,x^1,x^2):=(u(t,x^1),u(t,x^2))^\top$. Under our hypotheses, $U$ and $\mathbb{V}$ satisfy the conditions of Lemma~\ref{lemma:compactness}, thus there exists a weak solution $(Y,V)=((X^1,X^2)^\top,(W,W)^{\top})$ to the SDE 
    \begin{equation}
    \label{eq:Y-sde}
    \begin{cases}
        dY_t=U(t,Y_t)\,dt+ dV_t,\\
        Y_0=(y,y)^\top,
    \end{cases}
    \end{equation}
    such that (up to taking a $u$-dependent subsequence) $(Y^\ell,V)$ converges weakly in law to $(Y,V)$. We note that by the definition of $U$ and $\mathbb{V}$, letting $\coupm$ be the law of $(X^1,X^2,W)$, then both $\coupm(C^0([0,1]),dX^2,dW)$ and $\coupm(dX^1, C^0([0,1]),dW)$ are weak solutions to the SDE~\eqref{eq:main}. Thus to show Item~\ref{item:high prob pathwise}, we just need to control $\coupm(X^1|_{[0,\ep]} = X^2|_{[0,\ep]}).$
    
    To that end, we note that by the weak convergence in law of $Y^{\ell}$ to $Y$,
    \[\coupm(X^1|_{[0,\ep]}=X^2|_{[0,\ep]})\leq \coupm(|X^1_{T}-X^2_{T}|< \delta) \leq    \liminf_{\ell \to \infty}\W(|X^{1,\ell}_{T}-X^{2,\ell}_{T}|<\delta ) \leq \gamma,\]
    thus concluding the proof of Item~\ref{item:high prob pathwise}.

    We next prove Item~\ref{item:as non determin}. Keeping $u$ fixed in a full measure set as above, the previous argument gives for each $k \in \N$, a joint law $\coupm^k(dX^1, dX^2, dW)$ so that the marginals $\coupm^k(C^0([0,1]), dX^2, dW)$ and $\coupm^k(dX^1, C^0([0,1]), dW)$ are weak solutions to~\eqref{eq:main} and such that $\coupm^k(X^1|_{[0,\tau^k]} = X^2|_{[0,\tau^k]}) \leq 1/k$ for some $\tau^k \to 0$. For each $k \in \N$, let $\solm^{1,k}(dX \cond W)$ and $\solm^{2,k}(dX \cond W)$ be the conditional measures given $W$ of $X^{1,k}$ and $X^{2,k}$ respectively.

    Let 
    \[\solm^k(dX \cond W) := \frac{1}{2} \solm^{1,k}(dX \cond W) + \frac{1}{2} \solm^{2,k}(dX \cond W).\]
    We next claim that with $\W$ probability greater than $1-1/k$, the measure $\pi^{\tau^k}_*\solm^k(d\tilde X \cond W)$ is supported on more than one path. To see this, we note that if for some fixed $W$, $\pi^{\tau^k}_* \solm^k(d\tilde X \cond W)$ is supported on a single path, then it must be that $\pi^{\tau^k}_* \solm^{1,k}(d\tilde X \cond W)$ has the same support as $\pi^{\tau^k}_* \solm^{2,k}(d\tilde X \cond W)$. Then, for that $W$, we must have that $X^1|_{[0,{\tau^k}]} = X^2|_{[0,{\tau^k}]}$, but this happens with probability less than $1/k,$ thus giving the claim.

    We now define the measures
    \begin{equation}
    \label{eq:conditionals}
    \solm(dX \cond W) := \sum_{k=1}^\infty 2^{-k}\solm^k(dX \cond W) \quad \text{and} \quad  \solm(dX, dW) := \solm(dX \cond W) \W(dW).
    \end{equation}
    We note that for any $\ep>0$, $\pi^\ep_* \solm(d\tilde X \cond W)$ is supported on the union of the supports of $\pi^\ep_* \solm^k(d\tilde X \cond W)$, hence by the claim above is $\W$-almost surely supported on more than one path, since $\tau^k$ is eventually less than $\ep$. Thus to conclude the proof of Item~\ref{item:as non determin}, we just need to see that $\solm(dX,dW)$ defines a weak solution to~\eqref{eq:main}, for which we use Lemma~\ref{lem:disintegration}. It is clear that $\solm(C^0([0,1]), dW) =\W(dW)$ by definition. That $(X,W)$ solves~\eqref{eq:main} $\solm$ almost surely is also clear, since $\solm$ is built from other weak solution measures. The only condition to verify is that for all $t \in [0,1]$, $W\mapsto\pi^t_*\solm(d\tilde{X}\cond W)$ is $\overline{\sigma(W|_{[0,t]})}^{\mathbb{W}}$ measurable. This is, however, clear by construction and Lemma~\ref{lem:disintegration}, since the marginals of $\coupm^k$ are weak solutions, we have for all $j \in \{1,2\}, k \in \N,$ $W\mapsto\pi^t_*\solm^{j,k}(d\tilde X \cond W)$ is $\overline{\sigma(W|_{[0,t]})}^{\mathbb{W}}$ measurable. We thus conclude the proof of Item~\ref{item:as non determin}.

    Finally, we prove Item~\ref{item:weak unique as pathwise nonunique}. We continue letting $u$ be in the same full measure set and now suppose the~\eqref{eq:main} admits a unique weak solution $\solm(dX, dW)$, which we disintegrate into conditionals, $\solm(dX \cond W) \W(dW)$. We fix $\ep>0$.
    
    We first claim that for $\W$-a.e. $W$, $\pi^\ep_* \solm(d\tilde X \cond W)$ does not have an atom with mass larger than $1/2$. To see this, suppose for the sake of contradiction that with $\W$-probability greater than $\delta>0$, $\pi^\ep_* \solm(d\tilde X \cond W)$ has an atom with mass greater than $1/2+\delta$. By Item~\ref{item:high prob pathwise} and weak uniqueness, there exists a measure $\coupm(dX^1, dX^2, dW)$ such that $\coupm(dX, C^0([0,1]), dW) = \coupm(C^0([0,1]), dX,dW) = \solm(dX,dW)$ and $\coupm(X^1|_{[0,\ep]} = X^2|_{[0,\ep]}) \leq \delta^2$.

    However---conditionally on a set of $\W$ measure greater than $\delta$---$\pi^\ep_*\solm(d\tilde X \cond W)$ has an atom with mass at least $1/2+\delta$, so we must have that $X^1|_{[0,\ep]} = X^2|_{[0,\ep]}$ with $\coupm( dX^1, dX^2 \cond W)$-probability at least $2\delta$. Thus in total, we have that
    \[\coupm(X^1|_{[0,\ep]} = X^2|_{[0,\ep]}) \geq 2\delta^2 > \delta^2,\]
    which gives the desired contradiction and proves the claim.

    In order to conclude we apply Lemma~\ref{lem:coupling no diagonal} to $\pi^\ep_* \solm(d\tilde X \cond W)$ to get a measurable coupling $\coupm^\ep(d\tilde X^1, d\tilde X^2 \cond W)$ that charges no mass to the diagonal: $\coupm^\ep\big(\{(\tilde X, \tilde X) : \tilde X \in C^0([0,\ep])\}\big)=0$. Thus to conclude we just need to extend the coupling to all of $[0,1]$, getting a coupling $\coupm(dX^1, dX^2, dW)$ such that $\pi^\ep_* \coupm(d\tilde X^1, d \tilde X^2 \cond W) = \coupm^\ep(d\tilde X^1, d \tilde X^2 \cond W)$ (while maintaining the correct marginalization properties of $\coupm$). This is however straightforwardly done by disintegration.
\end{proof}

\appendix

\section{Weak solution theory}\label{appen:weak_sol_theory}

The purpose of this appendix is to prove that the disintegration-based definitions of weak and strong solutions given in Lemma~\ref{lem:disintegration} are equivalent to the classical definitions, as well as the Yamada--Watanabe Theorem as given in Lemma~\ref{lem:Yamada-Watanabe} and the compactness of weak solutions as given in Lemma~\ref{lemma:compactness}. We also provide the proof of Lemma~\ref{lem:coupling no diagonal}.

Throughout, we will have our drift field $u \in L^1_t C^0_x$ fixed. $(X,W)$ denote the canonical coordinates of a solution-noise pair, with $\W$ denoting the law of the driving noise and $\solm$ denoting the law of the pair. We let $\mathbb{E}$ denote the expectation with respect to $\solm$. For $t\in[0,T]$ we recall that $\pi^t$ denotes the projection map $\pi^t\gamma(s)=\gamma(s)$ for $s\in[0,t]$. We thus have that 
\[\pi^tX=X|_{[0,t]}\quad\text{and}\quad \pi^tW=W|_{[0,t]}.\]

We now prove Lemma~\ref{lem:disintegration}.

\begin{proof}[Proof of Lemma~\ref{lem:disintegration}]
We first prove Item~\ref{item:weak disintegration}. To this end, we fix $t\in[0,T]$. By definition, $\solm$ is a weak solution if and only if $\pi^tX$ and $W$ are conditionally independent given $\pi^tW$. This is equivalent to saying that for all bounded and continuous functions $f:C^0([0,t])\rightarrow \R$ and $g:C^0([0,T])\rightarrow \R$ it holds that
\begin{align*}
\mathbb{E}[f(\pi^tX)g(W)]=\mathbb{E}\big[\mathbb{E}[f(\pi^tX)g(W)\cond \pi^tW]\big]&=\mathbb{E}\big[\mathbb{E}[f(\pi^tX)\cond \pi^tW]\mathbb{E}[g(W)\cond \pi^tW]\big]
\\&=\mathbb{E}\Big[\mathbb{E}\big[ \mathbb{E}[f(\pi^tX)\cond \pi^tW]g(W)\cond \pi^tW\big]\Big]
\\&=\mathbb{E}\big[\mathbb{E}[f(\pi^tX)\cond \pi^tW]g(W)\big],
\end{align*}
where the third equality uses that $\mathbb{E}[f(\pi^tX)\cond \pi^tW]$ is $\sigma(\pi^tW)$ measurable. Using the definition of the conditional expectation, we always have that 
\[\mathbb{E}\big[\mathbb{E}[f(\pi^tX)\cond W] g(W)\big]=\mathbb{E}[f(\pi^tX)g(W)].\]
Thus the conditional independence is equivalent to
\[\mathbb{E}\big[\mathbb{E}[f(\pi^tX)\cond W] g(W)\big]=\mathbb{E}\big[\mathbb{E}[f(\pi^tX)\cond \pi^tW]g(W)\big],\]
for all suitable $f$ and $g$, and thus that $\mathbb{E}[f(\pi^tX)\cond W]=\mathbb{E}[f(\pi^tX)\cond \pi^tW]$ $\mathbb{W}$-almost surely.

Now, let $\solm(dX,dW)=\solm(dX\cond W)\W(dW)$ be a disintegration of $\solm$ with respect to $W$ so that for any $f$ as above
    \[\mathbb{E}[f(\pi^tX)\cond W]=\int_{C^0([0,t],\R^d)}  f(\tilde{X})\pi^t_*\solm(d\tilde{X}\cond W)\quad \mathbb{W}\text{-a.s.}\]
Since $\mathbb{E}[f(\pi^tX)\cond \pi^tW]$ is $\sigma(\pi^tW)$ measurable, the conditional independence of $\pi^tX$ and $W$ from $\pi^tW$ is thus equivalent to the function
\[W\mapsto\int_{C^0([0,t],\R^d)}  f(\tilde{X})\pi^t_*\solm(d\tilde{X}\cond W)\]
being $\overline{\sigma(\pi^tW)}^\mathbb{W}$-measurable for all $f$. This is in turn equivalent to $W\mapsto\pi^t_*\solm(d\tilde{X}\cond W)$ being $\overline{\sigma(\pi^tW)}^\mathbb{W}$ measurable as claimed.

We now prove Item~\ref{item:strong disintegration}. Suppose that $\solm$ is a strong solution so that $X$ is $\overline{\sigma(W)}^\solm$ measurable. This is equivalent to the existence of a measurable map $\Phi:C^0([0,T])\rightarrow C^0([0,T])$ such that $X=\Phi(W)$ $\solm$-almost surely, which is in turn equivalent to $\solm(dX\cond W)=\delta_{\Phi(W)}(dX)$ for $\mathbb{W}$-a.e. $W$. Since $W\mapsto \solm(dX\cond W)$ is a measurable map, this is in turn equivalent to $\solm(dX\cond W)$ being supported on a singleton $\mathbb{W}$-almost surely.
\end{proof}

Next we show Lemma~\ref{lem:Yamada-Watanabe}, which is equivalent to the Yamada--Watanabe Theorem~\cite{yamada_uniqueness_1971}.

\begin{proof}[Proof of Lemma~\ref{lem:Yamada-Watanabe}]
Let us first suppose that the SDE satisfies pathwise uniqueness. Then, suppose that $\solm$ and $\tilde\solm$ are any two weak solutions with respective disintegrations $\solm(dX \cond W)$ and $\tilde\solm(dX\cond W)$ with respect to $W$. Then let 
\[\coupm(dX^1,dX^2,dW)=\solm(dX^1\cond W)\tilde\solm(dX^2\cond W)\mathbb{W}(d W),\]
so that the marginal of $\coupm$ over $(X^1, W)$ is equal to $\solm$, and the marginal of $\coupm$ over $(X^2, W)$ is equal to $\tilde{\solm}$. Pathwise uniqueness thus implies that $X^1=X^2$ $\coupm$-almost surely, or equivalently $\W$-a.s.,
\[X^1=X^2,\qquad \solm(dX^1\cond W)\tilde\solm(dX^2\cond W)\text{-a.s.}\]
This then implies that
\[\solm(dX^1\cond W)\tilde\solm(dX^2\cond W)=\delta_{\Phi( W)}(dX^1)\delta_{\Phi( W)}(dX^2),\qquad\mathbb{W}\text{-a.s.}\]
for some $\Phi:C^0([0,T])\rightarrow C^0([0,T])$, which can be taken to be measurable due to the measurability of the conditional measure $\solm(dX^1 \cond W)$. We thus see that it must be the case that $\solm=\tilde\solm$: hence we have weak uniqueness and that both are strong solutions, allowing us to conclude this direction.

The other direction is almost immediate since if the SDE has weak uniqueness and any solution is a strong solution, then there must exist a measurable function $\Phi:C^0([0,T])\rightarrow C^0([0,T])$ such that for a weak solution $\solm$, we have that $\solm(dX,dW)=\delta_{\Phi(W)}(dX)\mathbb{W}(dW).$ Thus, if $\coupm$ is the joint law of two weak solutions, it must be the case that
\[\coupm(dX^1,dX^2,d W)=\delta_{\Phi( W)}(dX^1)\delta_{\Phi( W)}(dX^2)\mathbb{W}(d W),\]
i.e.\ $X^1=X^2$ $\coupm$-almost surely.
\end{proof}

We finally prove the compactness Lemma~\ref{lemma:compactness} for weak solutions. We recall that $\mathbb{W}$ is a regular driving noise, as defined in Definition~\ref{def:regular_driving_noise}.

\begin{proof}[Proof of Lemma~\ref{lemma:compactness}]
Suppose that $\mathbb{W}$ is a regular driving noise. We first note that since $C^0([0,T])$ is a Polish space, for all $\eps>0$ there exists a compact subset $K^\eps$ such that $\mathbb{W}(K^\eps)\geq 1-\eps$. Additionally, the Arzel\`a--Ascoli characterization of compact sets implies that $K^\eps$ is a bounded and uniformly equicontinuous family of functions. For all $\ep>0$, let $S^\ep \subseteq C^0([0,T])$ be defined by
\[S^\ep := \Big\{X : t \geq \tau,\, X_t= z + \int_\tau^t u(s,X_s)\,ds +  W_t - W_\tau; t \leq \tau,\,X_t = z; |z-y| \leq 1; \tau \in [0,T]; W \in K^\ep\Big\}.\]
One can readily verify that $S^\ep$ is closed, bounded, and uniformly equicontinuous, and thus compact, using that $K^\eps$ is a bounded and uniformly equicontinuous family and that $u \in L^1_t C^0_x$. This implies that $S^\eps\times K^\eps$ is compact as well. By our hypotheses on the support of $\solm^n$, for $n$ large enough,
\[\solm^n(S^\eps\times K^\eps)= \mathbb{W}(K^\eps),\]
thus the $\solm^n$ form a tight family of measures. There thus exists $\solm\in\mathcal{P}(C^0([0,T]) \times C^0([0,T]))$ such that $\solm^n\rightarrow \solm$ weakly up to a subsequence.

    Our goal is now to show that $\solm$ is a weak solution for the SDE~\eqref{eq:SDE_limit}, that is that it satisfies Items~\ref{item:marginalizes correct}-\ref{item:solves equation}. Item~\ref{item:marginalizes correct} is immediate as the marginal laws of $\solm^n$ are all equal to $\mathbb{W}$. Item~\ref{item:solves equation} also follows from the fact that any element $(X,W)$ of the support of $\solm$ has a sequence $(X^n,W^n)$ such that $(X^n,W^n)$ is in the support of $\solm^n$ and $(X^n,W^n) \to (X,W)$ in $C^0([0,T]) \times C^0([0,T])$. Using the integral equation, one can then readily see that Item~\ref{item:solves equation} follows.    
    
    All that remains is to conclude that $\solm$ satisfies Item~\ref{item:non-anticipatory}. By the same logic as that used in the proof of Lemma~\ref{lem:disintegration}, this is equivalent to showing that for all bounded and continuous $f:C^0([0,t])\rightarrow \R,g:C^0([0,T])\rightarrow \R,$ and $t\geq 0$,
\begin{equation}\label{eq:integrated_nonanticipatory}
      \mathbb{E}\big[f(\pi^tX)g(W)\big]=\mathbb{E}[f(\pi^tX)\mathbb{E}[g(W)\cond\pi^tW]\big],
    \end{equation}
    where the expectations throughout---consistent with the notation of this section---are being taken with respect to the measure $\solm$. Letting $G(\pi^t W):=\mathbb{E}[g(W)\cond\pi^tW]$, since $\mathbb{W}$ is a regular driving noise, $G$ is bounded and continuous, and the above reads
    \[\mathbb{E}\big[f(\pi^tX)g(W)\big]=\mathbb{E}\big[f(\pi^tX)G(\pi^t W)\big].\]
    Since~\eqref{eq:integrated_nonanticipatory} holds with $\solm$ replaced by $\solm^n$ for all $n\geq 1$ and $f(\pi^tX)g(W)$ and $f(\pi^tX)G(\pi^t W)$ are both bounded continuous functions of $(X,W)$, after taking limits, we find that~\eqref{eq:integrated_nonanticipatory} holds for $\solm$ as well.
\end{proof}

We finally prove Lemma~\ref{lem:coupling no diagonal}.

\begin{proof}[Proof of Lemma~\ref{lem:coupling no diagonal}]
    Since $\mathcal{Y}$ is Borel isomorphic to $[0,1]$~\cite[Corollary 6.8.8]{bogachev_measure_2007}, we can without loss of generality take $\mathcal{Y}=[0,1].$ Then for any measure $\mu \in \mathcal{P}([0,1])$, let $Q^\mu : [0,1] \to [0,1]$ denote the quantile map $Q^\mu(p) := \inf \{ t \in [0,1]: \mu([0,t]) \geq p\}$. Letting $\lambda$ denote the Lebesgue measure on $[0,1]$, we recall $Q^\mu_* \lambda = \mu.$ Let $\phi: [0,1) \to [0,1)$ be given by $\phi(y) = y + \frac{1}{2} \mod 1$. Then let $\lambda^2 \in \mathcal{P}([0,1]^2)$ be given by
    \[\lambda^2(dy^1,dy^2) := \int \delta_{z}(dy_1)\delta_{\phi(z)}(dy^2)\,dz,\]
    and $Q^{\mu,2}(y^1,y^2) = (Q^\mu(y^1), Q^\mu(y^2))$.
    Finally, let $\Gamma : \mathcal{P}([0,1]) \to \mathcal{P}([0,1]^2)$ be given by
    \[\Gamma(\mu) := Q^{\mu,2}_* \lambda^2.\]
    We then readily see that $\Gamma(\mu)(dy, [0,1]) = \Gamma(\mu)([0,1],dy) = \mu(dy).$ It is also clear that $\Gamma$ is a Borel measurable map. We thus let $\gamma_x := \Gamma(\mu_x)$.

    In order to conclude, we just need to see that if $\mu$ has no atoms with mass greater than $1/2$, then $\Gamma(\mu)(\{(y,y) : y \in [0,1]\}) =0.$ Note that by construction
    \[\Gamma(\mu)(\{(y,y) : y \in [0,1]\}) = \P\big(Q^\mu(Z) = Q^\mu(\phi(Z))\big),\]
    where $Z$ has the uniform distribution on $[0,1].$ Suppose that there is a positive measure set of such $Z$, then (up to interchanging $z$ and $\phi(z)$) there must be $ 0 \leq z^1< z^2<1/2$ distinct such that $Q^\mu(z^i) = Q^\mu(\phi(z^i))$ for $i =1,2$. Then since $Q^\mu$ is increasing, we must have that $Q^\mu(z^1) = Q^\mu(z^1+1/2) = Q^\mu(z^2)  =Q^\mu(z^2+1/2) =:t$. Then we have that $\mu([0,t]) \geq z^2+1/2$ and $\mu([0,t)) \leq z^1$. Thus $\mu(\{t\}) \geq 1/2 + z^2 - z^1 >1/2$, contradicting our hypothesis. Thus we get the claim and so conclude the proof.
\end{proof}

{\small
\bibliographystyle{alpha}
\bibliography{cleanreferences}
}

\end{document}